\newcounter{myctr}

\documentclass{ws-m3as-arxiv}
\usepackage{url,subfigure}
\usepackage{lebListes}  %%% modifie les bullets selon AMS et fournit les macros suivants
\noitemsep  %%% no spacing between list items
\usepackage{algorithm,algpseudocode}

\usepackage{color,pstcol,pstricks}
\newgray{backgroundgray}{0.90}
\definecolor{backgroundgray}{gray}{0.90}

\newcommand{\eqdef}{\stackrel{\mathrm{def}}{=}} % by definition

\newenvironment{petit}
{\par\vspace{.5\baselineskip}\noindent\footnotesize}
{\nobreak\par\vspace{.5\baselineskip}}

\newcommand{\bp}{ % begin program
  \small \ttfamily
 \begin{tabbing}
 aaa\=aaa\=aaa\=aaa\=aaaaaaaa \= aaaaaaaaaa\= \kill
 }
\newcommand{\ep}{\end{tabbing}\normalfont\normalsize }% end program

\newcommand{\pro}[1]{{\ttfamily\small#1}}

\newcommand{\fref}[1]{Figure~\ref{#1}}
\newcommand{\sref}[1]{Section~\ref{#1}}
\newcommand{\cref}[1]{Chapter~\ref{#1}}
\newcommand{\coref}[1]{Corollary~\ref{#1}}
\newcommand{\eref}[1]{Eq.(\ref{#1})}

\newcommand{\thref}[1]{Theorem~\ref{#1}}
\newcommand{\pref}[1]{Proposition~\ref{#1}}

\newcommand{\calA}{\mathcal A}

\newcommand{\calC}{\mathcal C}

\newcommand{\calL}{\mathcal L}

\newcommand{\calP}{\mathcal P}

\newcommand{\calS}{\mathcal S}

\newcommand{\lp}{\left(}
\newcommand{\lb}{\left[}
\newcommand{\lc}{\left\{}

\newcommand{\rp}{\right)}
\newcommand{\rb}{\right]}
\newcommand{\rc}{\right\}}

\newcommand{\eps}{\varepsilon}

\newcommand{\abs}[1]{\left| #1 \right|}

\newcommand{\ind}[1]{1_{\{#1\}}}

\def\be{\begin{equation}}
\def\ee{\end{equation}}
\def\ben{\[}
\def\een{\]}
\def\bearn{\begin{eqnarray*}}
\def\eearn{\end{eqnarray*}}
\def\bear{\begin{eqnarray}}
\def\eear{\end{eqnarray}}
\def\barr{\begin{array}}
\def\earr{\end{array}}
\def\bmat{\left(\begin{array}}
\def\emat{\end{array}\right)}

\newcommand{\limit}[2]{\lim_{#1 \rightarrow #2}}

\newcommand{\mif}{\mathrm{\; if \; }}

\newcommand{\mfor} {\mathrm{\;  for \;  }}

\def\Reals{\mathbb{R}}

\def\Nats{\mathbb{N}}
\def\E{\mathbb{E}}
\def\P{\mathbb{P}}

\newcommand{\cro}[1]{\langle#1\rangle}   % bracket for integrals etc

\begin{document}

\makeatletter
\def\@biblabel#1{[#1]}
\makeatother

\markboth{G\'omez-Serrano, Graham and Le Boudec}{Bounded Confidence
Model Of Opinion Dynamics}

%%%%%%%%%%%%%%%%%%%%% Publisher's Area please ignore %%%%%%%%%%%%%%%
%
\catchline{}{}{}{}{}
%
%%%%%%%%%%%%%%%%%%%%%%%%%%%%%%%%%%%%%%%%%%%%%%%%%%%%%%%%%%%%%%%%%%%%

\title{The Bounded Confidence Model Of Opinion Dynamics
}

\author{\footnotesize G\'OMEZ-SERRANO, JAVIER}

\address{Instituto de Ciencias Matem\'aticas, Consejo Superior de Investigaciones Cient\'ificas \\(ICMAT CSIC-UAM-UCM-UC3M),
Serrano 123, 28006 Madrid, Spain, javier.gomez@icmat.es \footnote{J. G\'omez-Serrano's research was done while being an exchange student at EPFL.}}

\author{\footnotesize GRAHAM, CARL}

\address{Centre de math\'ematiques appliqu\'ees, CNRS et \'Ecole Polytechnique \\
91128 Palaiseau, France,
carl@cmap.polytechnique.fr}

\author{LE BOUDEC, JEAN-YVES}

\address{EPFL, 1015 Lausanne, Switzerland\\jean-yves.leboudec@epfl.ch}

%\author{\footnotesize FIRST AUTHOR\footnote{Typeset names in
%10~pt Times Roman, uppercase. Use the footnote to indicate
%the present or permanent address of the author.}}

%\address{University Department, University Name, Address\\
%City, State ZIP/Zone,
%Country\footnote{State completely without abbreviations, the
%affiliation and mailing address, including country. Typeset in 8~pt
%Times italic.}\\
%first\_author@university.edu}

\maketitle

%\begin{history}
%\received{(received date)}
%\revised{(revised date)}
%%\accepted{(Day Month Year)}
%%\comby{(xxxxxxxxxx)}t
%\end{history}
\begin{abstract}
The bounded confidence model of opinion dynamics, introduced by Deffuant \emph{et al}, is a 
stochastic model for the evolution of continuous-valued opinions within a finite group of peers.
We prove that, as time goes to infinity, the opinions evolve globally into 
a random set of clusters too far apart to interact, and thereafter all opinions in every cluster
converge to their barycenter.
We then prove a mean-field limit result, propagation of chaos:
as the number of peers goes to infinity 
in adequately started systems
and time is rescaled accordingly,
the opinion processes converge to i.i.d. nonlinear Markov (or McKean-Vlasov) processes;
the limit opinion processes evolves as if under the influence of opinions drawn from
its own instantaneous law, which are the unique solution of a 
nonlinear integro-differential equation of Kac type.
This implies that
the (random) empirical distribution processes converges to this (deterministic) solution.
We then prove that, as time goes to infinity, this solution converges to a law concentrated on isolated opinions too far apart to interact,
and identify sufficient conditions for the limit not to depend on the
initial condition, and to be concentrated at a single opinion.
Finally, we prove that if the equation has an initial condition with a density,
then its solution has a density at all times, develop
a numerical scheme for the corresponding functional equation,
and show numerically that bifurcations may occur.
\end{abstract}

\keywords{Social networks; reputation; opinion; mean-field limit;
propagation of chaos; nonlinear integro-differential equation;
kinetic equation; numerical experiments.}

\ccode{MSC2010: 91D30,60K35,45G10,37M99}

%\keywords{Social networks; reputation; mean field;
%integro-differential equations.}

\section{Introduction}

Some models about opinion dynamics (or belief or gossip propagation, etc.) are based on binary
values,\cite{follmer1974random,arthur1994increasing,orlean1995bayesian,latane1997self,weisbuch2002dynamical,sznajd2000opinion}
and often lead to attractors that display uniformity of opinions.
These models are not valid for
scenarios such as the social network of truck drivers interested
in the quality of food of a highway restaurant or the critics'
ratings about the new opening movies, for which it is required to have
a continuous spectrum of opinions, as is also the case in politics when people
are positioned on a scale going from extreme left-wing to right-wing opinions.\cite{1957}

The bounded confidence model introduced by Deffuant \emph{et al}.\cite{deffuant2000mixing}
is a popular model for such scenarios.
Peers have $[0,1]$-valued opinions;
repeatedly in discrete steps, two peers are sampled, and if their opinions differ by at most a \emph{deviation threshold} then both move closer, in  barycentric fashion
governed by a \emph{confidence factor}. These parameters are the same for all peers, and the system is in binary mean-field interaction.
The model has been studied and generalized, notably
to other interaction graphs than the fully-connected one,\cite{weisbuch2003interacting,weisbuch2004bounded,weisbuch2005persuasion,dittmer2001consensus,hegselmann2002opinion}
to vector-valued opinions,\cite{neau2000revisions,weisbuch2003interacting}
and to peer-dependant deviation thresholds\cite{weisbuch2005persuasion}.

Reputation systems have lately emerged due to the necessity to
measure trust about users while doing transactions over the internet;
popular examples can be found in
e-Bay\cite{resnick2002trust} or Bizrate.\cite{wang2004design}
Some models for trust evolution and the potential effects of groups
of liars attacking the system can be seen as a
generalization of the bounded confidence model,
in particular
when there are no liars nor direct observations
and the system evolves only by interaction between the peers.\cite{le2007generic,buchegger2004robust}
The ``Rendez-vous" model used by Blondel et al \cite{blondel2r} has qualitative resemblance to the model used in this paper; like ours, it converges to a finite number of clusters in finite time for the finite $N$ case. However, the interaction model is different, and our techniques (based on convexity and conservation of mean, see \pref{prop-convex}) do not seem to apply to this model.

The mean-field approximation method for large interacting systems
has a very long history. Its heuristic and rigorous use started in statistical
physics\cite{sznitman1991topics,graham1997stochastic,meleard1996asymptotic,bellomo2000modeling}
and entered many other fields, notably communication
networks,\cite{whitt1985blocking,kelly1991loss,graham1993propagation,graham1994chaos,graham2000chaoticity,graham2000kinetic,1555363}
TCP connections,\cite{baccelli2002mean,tinnakornsrisuphap2003limit,baccelli2006http,graham2009interacting}
robot swarms,\cite{martinoli2004modeling} transportation
systems,\cite{afanassieva1997models} and online reputation
systems\cite{le2007generic,mundinger2005analysis,mundinger-impact} in which is particularly appealing
since the number of users may be very large (over 400 million for
Facebook\cite{facebookstatistics}).

%The mean-field approximation method for large interacting systems
%has originated in statistical mechanics, notably after the seminal work of
%Ludwig Boltzmann.\cite{sznitman1991topics,graham1997stochastic,meleard1996asymptotic,desvillettes1999probabilistic,bellomo2000modeling}
%It has been used, heuristically and rigorously,
%in many other fields,
%notably communication
%networks,\cite{whitt1985blocking,kelly1991loss,graham1993propagation,graham1994chaos,graham2000chaoticity,graham2000kinetic,1555363}
%TCP connections,\cite{baccelli2002mean,tinnakornsrisuphap2003limit,baccelli2006http,graham2009interacting}
%robot swarms,\cite{martinoli2004modeling} transportation
%systems,\cite{afanassieva1997models} and online reputation systems\cite{le2007generic,mundinger2005analysis,mundinger-impact} in which is particularly appealing
%since the number of users may be very large (over 400
%million for Facebook\cite{facebookstatistics}).

%%%

This paper provides some rigorous proofs of old and new results
on the Deffuant \emph{et al}.\cite{deffuant2000mixing} model, which has been studied intensively, but essentially by heuristic arguments and simulations.
Notably,
justifying the validity of the mean-field approach is not a simple matter,
and classical methods do not apply,
as seen below.

We prove that as time increases to infinity,
opinions eventually group after some random finite time into a constant number of clusters,
which are separated by more than the deviation threshold, and cannot influence one another.
Thereafter, all opinions within every such cluster converge to their barycenter.
The limit distribution of opinions is thus of a degenerate form, in which there are only a
small number of fixed opinions which differ too much to influence each
other, called a ``partial consensus"; when it is constituted
of one single opinion, it is called a ``total consensus''. Note that the limit distribution is itself random, i.e. different sample runs of the same model with same initial conditions always converge, but perhaps to different limiting distributions of opinions.

We then prove a mean-field limit result, called ``propagation of chaos'' in statistical mechanics:
if the number of peers goes to infinity, the systems are adequately started,
and time is rescaled accordingly, then the processes of the opinions converge in law to i.i.d. processes.
Each of these is a so-called nonlinear Markov (or McKean-Vlasov) process, corresponding to
an opinion evolving under the influence of opinions drawn
independently from the marginal law of the opinion process itself, at a rate which
is the limit of that at which a given peer in the finite system encounters its peers.
Moreover, these marginals are the unique solution of an adequately started
nonlinear integro-differential equation.

This implies a law of large numbers: the empirical measures of the interacting processes
converge to the law of the nonlinear Markov process.
Such process level results imply results for the marginal laws,
but they are much stronger: limits are derived for functionals of the sample paths,
such as hitting times or extrema. In particular, a functional law of large numbers holds for the marginal processes of these empirical measures, with limit the solution of the
integro-differential equation.

The probabilistic structure of this limit equation is similar to that of kinetic equations
such as the cutoff spatially-homogeneous
Boltzmann or Kac equations, classically used in statistical mechanics to describe
the limit of certain particle systems with binary interaction.
Under quite general assumptions, satisfied here, it has long been known that
it is well-posed in the space of probability laws,
and that if the initial law has a density, then the solution has a density at all times
satisfying a functional formulation of this equation.

\begin{remark}
\label{r-smf}
There are two main difficulties in the propagation of chaos proof:
\begin{enumerate}
\item
\label{dif1}
the interaction is \emph{binary} mean-field, since two opinions change simultaneously,
\item
\label{dif2}
the indicator functions related to
the deviation threshold are discontinuous.
\end{enumerate}
A  system in which only one opinion would change at a time would be
in \emph{simple} mean-field interaction, and one could write equations for the opinions
in almost closed form, which could be passed to the limit in various classical ways.
This cannot be done for binary interaction, in which there is much more feedback between peers;
moreover, this would require continuous coefficients.
See \sref{s-diffmfpr} for details.
\end{remark}

Such difficulties have been solved before\cite{sznitman1991topics}.
In order to adapt results obtained for a class of interacting systems
inspired by communication network models\cite{graham1993propagation,graham1994chaos}
using stochastic coupling techniques,
which can be applied to various Boltzmann and Kac models\cite{graham1997stochastic,desvillettes1999probabilistic},
we introduce an intermediate \emph{auxiliary} system,
a continuous-time variant of the discrete-time model of Deffuant \emph{et al}.\cite{deffuant2000mixing} interacting
at Poisson instants, which itself constitutes a relevant opinion model.

For this auxiliary system,
we prove propagation of chaos, in total variation norm with estimates on any finite time interval.
We then control the distance between this auxiliary system and the
Deffuant \emph{et al}. model, and prove propagation of chaos for a weaker topology,
but still at the process level and allowing discontinuous measurable
dependence on the $[0,1]$-values taken by the opinions.

The method can be easily generalized, for instance to
vector-valued opinions, or to randomized interactions with a joint law
governing whether one or both peers change opinion and by how much; for instance,
choosing uniformly at random one peer to change opinion and leaving the other fixed
would lead to a simple mean-field interacting model,
and the limit model would be slowed down by a factor two.

To the best of our knowledge, this is the first rigorous mean-field limit result
for this model.
Similar integro-differential equations were used without formal justification
before,\cite{deffuant2000mixing,ben2003bifurcations,lorenz2007continuous}
and appear to be incorrect by a factor~2 (perhaps by disregarding that two peers change opinion at once), which illustrates the
interest of deriving the macroscopic equation from a
microscopic description, as we do here.

We thank a referee to have brought to our attention the preprint
Como-Fagnani.\cite{como2010scaling}
It contains results for the marginal laws of a continuous-time
variant of the model with two major simplifications:
the interaction is simple mean-field (only one opinion changes at a time),
and the indicator functions are replaced by Lipschitz-continuous functions; this
removes difficulties (\ref{dif1}) and (\ref{dif2})
in Remark~\ref{r-smf}, to which its techniques do not apply.
We have overcome these difficulties in the precise model of Deffuant \emph{et al}.,\cite{deffuant2000mixing}
and have given much stronger results,
for process laws in total variation norm and not for marginal laws in weak topologies.

One expects that the long-time behavior for the
mean-field limit should be highly related to the behavior for an large number of peers of
the long-time limit of the finite model.
This heuristic inversion of long-time and large-number limits can be sometimes rigorously justified, for instance by a compactness-uniqueness
method,\cite{whitt1985blocking,kelly1991loss,graham2000chaoticity,graham2000kinetic}
but here the limit nonlinear integro-differential equation may have multiple equilibria, and
formal proof would constitute a formidable task.

We prove that the long-time behavior of the solution of the limit integro-differential equation is similar
to that of the
model with finitely many peers: it converges  to a
partial consensus constituted of a small number of fixed opinions
which differ too much to influence each
other.

We then develop a numerical method for the limit equation, and use it
to explore the properties of the model. We observe phase transitions
with respect to the number of limit opinions,
while varying the deviation threshold for some fixed
initial condition.
We model the scenario of a company fusion, dividing the
workers into an ``undecided'' group and two ``extremist'' factions, and obtain that
having 20\% of the workers ``undecided'' is enough to achieve consensus between all.

Last, we
establish a bound on the deviation threshold, allowing
to determine if there is total consensus or not, under the
assumption of symmetric initial conditions.

%\subsection{Organisation of the paper}
In the sequel, \sref{sec-model} describes the finite model,
and \sref{sec-finite} studies some of its long time properties.
\sref{sec-convergence} rigorously derives the mean-field limit,
\sref{sec-infinite} studies some of its long time properties,
and \sref{sec-num} is devoted to numerical results. The appendix contains some
probabilistic complements in Section~A, the details of the algorithm in Section~B and all proofs in Section~C.
%
%\color{red}
%\LARGE I can't fix the references of the sections, but they need to be changed.
%\color{black}
%
%\normalsize
%\sref{sec-probab} contains some
%probabilistic complements,
%\aref{sec-algo} some details on the algorithms,
%and \sref{sec-proof} all proofs.
%

% and all tables are in \sref{sec-tables}.
%\medskip
%%The structure of the thesis is the following. In section 2 we
%summarize the notation used in the thesis. In section 3 we
%analyze the probabilistic model for a finite number of peers.
%In section 5 we apply a mean-field approach and we repeat the
%analysis for the deterministic system when the number of peers
%tends to infinity. The proofs between the finite and infinite
%system are done in section 4. We present a technique for
%bounding the critical value of $\Delta$ which represents the
%transition of the limit distribution between one Dirac and two
%Diracs in section 5.7. The numerical method and the proofs for
%its order of convergence and complexity can be found in section
%6 and the code in the appendix. Finally, in section 7 we
%present simulations that illustrate the behavior of the system:
%we present some of the different limit functions that the
%distribution might tend to and we also study the bifurcation
%diagrams for some initial conditions under certain scenarios.

% Proofs are in appendix.

\section{Interacting system model, and reduced descriptions}
\label{sec-model}

The model for $N\ge2$ interacting peers introduced by Deffuant \emph{et al}.\cite{deffuant2000mixing}
is as follows. The random variable (r.v.)
$X^N_i(k)$ with values in $[0,1]$ denotes the reputation record
kept at peer $i\in\{1,\ldots,N\}$ at time
$k\in\Nats=\{0,1,\dots\}$, representing its opinion (or belief, etc.)
about a given subject, the same for all peers.
The discrete-time process of the states taken by the system of peers is
\[
X^N = (X^N(k),{k\in\Nats})\,,
\qquad
X^N(k) = (X_i^N(k))_{1 \le i \le N}\,,
\]
and evolves in function of the \emph{deviation threshold}
$\Delta\in (0,1]$ and the \emph{confidence factor} $w\in (0,1)$.
At each instant $k$, two peers $i$ and $j$ are selected uniformly at random without replacement,
and:
\begin{itemize}
\item if $|X_i^N(k) - X_j^N(k)| > \Delta$ then $X^N(k+1)=X^N(k)$,
the two peers' opinions being too different for mutual influence,
   \item
if $|X_i^N(k) - X_j^N(k)| \le \Delta$ then  the values of peers $i$ and $j$
are updated to
\[
\left\{
\begin{aligned}
X_i^N(k+1) &= wX_i^N(k) + (1-w)X_j^N(k)\,,
\\
X_j^N(k+1) &= wX_j^N(k) + (1-w)X_i^N(k)\,,
\end{aligned}
\right.
\]
and the values of the other peers do not change at time $k+1$, the two peers
having sufficiently close opinions to influence each other.
\end{itemize}

Small values of $\Delta$ and
large values of $w$ mean that the peers trust very much their own opinions
in comparison to the new information given by the other
interacting peer.
The extreme excluded values $\Delta=0$ or $w = 1$ correspond to peers never changing opinion,
and $w = 0$ to peers switching opinions if close enough. For $w = 1/2$, two close-enough peers
would both end up with the average of their opinions.

A reduced, or macroscopic, description of the system is given by the
\emph{empirical measure} $\Lambda^N$,
and by its marginal process $M^N = (M^N(k), {k\in\Nats})$
also called the \emph{occupancy process},
\begin{equation*}
\Lambda^N = \frac1N \sum_{i=1}^N \delta_{X^N_i}
= \frac1N \sum_{i=1}^N \delta_{(X^N_i(k), {k\in\Nats})}\,,
\qquad
M^N(k) = \frac1N \sum_{i=1}^N \delta_{X^N_i(k)}\,.
\end{equation*}
The random measure $\Lambda^N$ has samples in $\calP([0,1]^\Nats)$, the space of probability
measures
on $[0,1]^\Nats$; its projection $M^N = (M^N(k), {k\in\Nats})$, which carries much less information,
has sample paths in $\calP([0,1])^\Nats$, the space of sequences of probability measures
on $[0,1]$.
For measurable $g: [0,1]^\Nats \to \Reals$ and $h: [0,1] \to \Reals$,
\[
\cro{g,\Lambda^N} = \frac{1}{N} \sum_{i=1}^N g(X_i^N)\,,
\qquad
\cro{h,M^N(k)}= \frac{1}{N} \sum_{i=1}^N h(X_i^N(k))\,.
\]

We will also re-scale time as $t= \frac{k}{N}$, and consider in particular the
rescaled occupancy process $\tilde{M}^N = (\tilde{M}^N(t),
{t\in\Reals_+})$ given by $
 \tilde{M}^N(t) = M^N(\lfloor N t\rfloor)$, which in  \sref{sec-convergence}
 will be shown to converge
to a deterministic process $(m(t),{t\in\Reals_+})$.
 %,process obtained when
%\begin{align*}
%\cro{g,\Lambda^N}&:=\int_{[0,1]^\Nats} g(x) \Lambda^N(dx) = \frac{1}{N} \sum_{n=1}^N g(X_i^N)\,,
%\\
%\cro{h,M^N(k)}&:=\int_{[0,1]} h(z) M^N(k)(dz) = \frac{1}{N} \sum_{n=1}^N h(X_i^N(k))\,.
%\end{align*}

%The rest of the paper is organized as follows.
%\sref{sec-finite} studies some long time properties of this model.
%\sref{sec-convergence} studies the mean-field asymptotic regime in which $N$ goes to infinity
%and time is rescaled as $ t= \frac{k}{N}$. In particular the rescaled
%occupancy process $\tilde{M}^N = (\tilde{M}^N(t), {t\in\Reals_+})$, given by
%$
% \tilde{M}^N(t) = M^N(\lfloor N t\rfloor)
%$,
% is shown to converge in probability
% to a deterministic process $(m(t),{t\in\Reals_+})$, called the ``mean field
% limit", satisfying an integro-differential equation.
%\sref{sec-infinite} studies some long time properties  of this mean-field limit.
%\sref{sec-num} is devoted to numerical results for $(m(t),{t\in\Reals_+})$, using an equivalent
%partial integro-differential equation formulation.
%\sref{sec-conc} contains some concluding remarks, and
%\sref{sec-proof} all proofs.

\section{Long-time behavior of the finite N model}
\label{sec-finite}

We consider a fixed finite number of peers and let time $k$ go to infinity. We prove that the distribution of peer opinions $M^N(k)$ converges almost surely (a.s.) to a random distribution $M^N(\infty)$. Note that the limiting distribution $M^N(\infty)$ depends on chance as well as on the initial condition. We prove that
$M^N(\infty)$ is a combination of at most
$\left\lceil\frac{1}{\Delta}\right\rceil$ Dirac measures at points
separated by at least $\Delta$. A key observation
here is that if $h$ is any convex function then $\cro{h,M^N(k)}$ is
non-increasing in $k$. Dittmer and Krause
\cite{dittmer2001consensus,krause2000discrete} obtained similar
results, but for a deterministic model.

\begin{definition}
\label{d-partcons}
We say that  $\nu\in \calP[0,1]$
is a \emph{partial consensus with $c$ components} if $\nu =
\sum_{m=1}^{c}\alpha_m \delta_{x_m}$ with
 $x_m \in [0,1]$, $\abs{x_m-x_{m'}}
> \Delta$ for $m \neq m'$, and $\alpha_m >0$. Necessarily
$c \leq \left\lceil\frac{1}{\Delta}\right\rceil$ and $\sum_{m= 1}^{c} \alpha_m
= 1$.
If $c=1$, \emph{i.e.}, if $\nu$ is a Dirac measure, we say that
$\nu$ is a \emph{total consensus}.
\end{definition}

If $M^N(k)$ is a partial consensus, then
peers are grouped in a number of components too far apart to interact,
and
within one component all peers have the same
value.
Thus, a partial consensus is an absorbing state for
$M^N$, and \thref{discrete_convergence}
below will show that $M^N(k)$ converges a.s., as
$k\to \infty$, to one such state.

\subsection{Convexity and Moments} We start with results about convexity and moments, which are
needed to establish the convergence result and are also of
independent interest.

\begin{proposition}For any convex function $h: [0,1] \to \Reals$,
any $x,y$ and $w$ in $[0,1]$,
 \ben
 h\lp wx+(1-w)y\rp+h\lp wy+(1-w)x\rp-h(x)-h(y)\leq 0
 \een
 with equality when $h$ is strictly convex possible only if
 $x=y$ or $w =0$ or $w=1$.\label{lem-convex}\label{prop-convex}
\end{proposition}

The following corollary is immediate from the interaction structure of the model.

\begin{corollary}
If $h: [0,1] \to \Reals$ is a convex function, then $\cro{h,M^N(k)}$ is a non-increasing function of $k$
along any sample path. \label{prop-fn-conv}
\end{corollary}

Applying this to $h(x)=x$, $h(x)=-x$ and $h(x)=x^n$ yields that in any sample path,
the first moment is constant and other moments are non-increasing with time.

\begin{corollary}
\label{cor-fin-mom}
For $n =1,2,\dots$ and $k\in\Nats$, let $\mu_n^N(k) =
\frac{1}{N}\sum_{i=1}^{N}X_i^N(k)^n$ denote the $n$-th moment
of $M^N(k)$, and let $\sigma^N(k)$ be the standard deviation given by
$\sigma^N(k)^2= \mu^N_2(k)-\mu^N_1(k)^2 $.
Then:
\begin{enumerate}
  \item The mean $\mu_1^N(k)$ is stationary in $k$, \emph{i.e.}, $\mu_1^N(k)=\mu_1^N(0)$
  for all $k$.
  \item The moments and the standard deviation are
      non-increasing in $k$: if $k \leq k'$
      then $\mu_n^N(k) \geq \mu_n^N(k')$ and $\sigma^N(k)
      \geq \sigma^N(k')$.
\end{enumerate}\label{discrete_moments}\label{discrete_variance_drop}
\end{corollary}
%
%Hence, the variance and the standard deviation
%are non-decreasing in time.
%
%\begin{corollary}
%\label{discrete_variance_drop}
%The standard deviation
%$\sigma^N(k)$ given by
%$\sigma^N(k)^2= \mu^N_2(k)-\mu^N_1(k)^2 $ is
%non-increasing in $k$, \emph{i.e.}, if $k \leq k'$ then $\sigma^N(k) \geq \sigma^N(k')$.
%%Moreover, if peers $i$ and $j$
%interact and $x_i(k)$ and $x_j(k)$ pass the deviation test, the
%drop on the variance is given by
%$\frac{1}{N^2}2w(1-w)(x_i(k)-x_j(k))^2$. Otherwise the variance
%is the same.
%\end{corollary}
%\begin{proof}
%The first part of the statement follows automatically from \ref{discrete_moments}. For the second part, we should notice that:
%$$ \sigma^N(t+1)^2 - \sigma^N(k)^2 = \mu_2^N(t+1) - \mu_2^N(k) - \mu_1^N(t+1)\mu_1^N(t+1) + \mu_1^N(k)\mu_1^N(k) = \mu_2^N(t+1) - \mu_2^N(k)$$
%where in the last equality we have used corollary \ref{discrete_constant_mean}. Particularizing equation (\ref{decreasing_discrete_moments}) to the case $k = 2$, we get that:
%$$ a_0 = a_2 = (1-w)^2+w^2-1 = -\frac{1}{N^2}2w(1-w)$$
%$$ a_1 = \frac{1}{N^2}4w(1-w)$$
%
%We can easily factor now:
%$$ \mu_2^N(t+1) - \mu_2^N(k) = \frac{1}{N^2}\left(4w(1-w)x_i(k)x_j(k)-2w(1-w)x_i(k)^2 - 2w(1-w)x_j(k)^2\right)
%= -\frac{1}{N^2}2w(1-w)(x_i(k) - x_j(k))^2$$
%
%Trivially, if the peers don't pass the deviation test, there is no change in the variance.
%\end{proof}

Moreover, stationarity of moments is equivalent to
reaching partial consensus:

\begin{proposition}
\label{pro-partcons}
If $M^N(k)$ is a partial consensus, then $\mu_n^N(k')=\mu_n^N(k)$ for all $n\ge1$ and $k'\geq k$.
Conversely, if for some $n \geq 2$ there exists a (random) instant $k$ such that $\mu_n^N(k')=\mu_n^N(k)$ for all $k'\geq k$,
then  $M^N(k')=M^N(k)$ for all $k'\geq k$
and $M^N(k)$ is a partial consensus, almost surely.
\end{proposition}

\subsection{Almost Sure Convergence to Partial Consensus}

\begin{definition}
We say that two peers $i$ and $j$ are
\emph{connected} at time $k$ if their values $x$ and $y$
satisfy
 $\abs{y-x} \leq \Delta$. We
say that $F\subset \{1,2,\dots,N\}$ is a \emph{cluster} at time
$k$ if it is a maximal connected component.
\end{definition}
In other words, a cluster is a maximal set of peers such that
every peer can pass the deviation test with one neighbour in
the cluster. The set of clusters at time $k$ is a random
partition of the set of peers.
The following proposition states that a cluster can either
split or stay constant, but cannot grow.

\begin{proposition}
\label{pro-clusters}
Let $\calC^N(k)=\{C_1, \dots ,C_{\ell}\}$
be the set of clusters at time $k$. Then either $\calC^N(k+ 1)=
\calC^N(k)$ or
  $\calC^N(k+ 1) = \lp \calC^N(k) \setminus C_{\ell_1}\rp \cup \calC'$
where $\calC'$ is a partition of $C_{\ell_1}$,
for some $\ell_1\in \{1, \dots, \ell\}$.
 \label{cluster_not_grow}
\end{proposition}
The number of clusters is thus non decreasing,
and since it is bounded by $\left\lceil
\frac{1}{\Delta} \right\rceil$ it must be
constant after some time, yielding the
following:

\begin{corollary}
\label{constant_clusters_after_time} There exists a random time $K^N$,
a.s. finite, such that
 \ben
 \calC^N(k) = \calC^N(K^N) \mfor k \geq K^N\,.
 \een
\end{corollary}

Finally, we prove that the occupancy measure
converges to a partial consensus (see \ref{sec-probab} for the
usual weak topology on $\calP[0,1]$):

\begin{theorem}
\label{thm-longtimelim}
\label{discrete_convergence} As $k$ goes to infinity,  $M^N(k)$ converges
almost surely, for the weak topology on $\calP[0,1]$, to a random probability $M^N(\infty)$,
which is a partial consensus with $L^N$ components,
where $L^N:=\textnormal{Card}(\calC^N(K^N))$ is
the \emph{final number of clusters}.
\end{theorem}

\thref{thm-longtimelim} notably implies that there is
convergence to total consensus if and only if $L^N = 1$. The
probability $p^*:=\P(L^N=1)$ of convergence to total consensus
is not necessarily $0$ or $1$, but:
%[More
%precisely\dots
% The following follows immediately from
%Theorem~1.
%\begin{corollary}
%Conditional to $L^N=1$, $M^N(k)$ converges almost
%surely, in the weak-* topology, as $k\to \infty$,
%to a total consensus. The number of
%\end{corollary}
%Say something about $p^*=\P(L^N=1)$. I
%conjecture that:
\begin{enumerate}
    \item If the diameter of $M^N(0)$ is less than $\Delta$
        (\emph{i.e.}, $\max_{i,j}\abs{X^N_i(0)-X^N_j(0)}<
        \Delta$) then $p^*=1$ (obvious);
    \item If there is more than 1 cluster in $M^N(0)$ then
        $p^*=0$ (see \pref{cluster_not_grow}).
   % \item If none of these are true, then
    %$0<p^*<1$ (to be done; there is a bounded number of steps that leads either to (1) and another set
    %of steps that leads to (2)).]
    %[FALSE. I think I have a counterxample, see below. Carl]
\end{enumerate}
%%%
%Let $N=3$ and $\Delta=\frac34$ and $w=\frac12$.
%If $X^3_1(0)=0$ and $X^3_2(0)=\frac12$ and $X^3_3(0)=1$ then there is one cluster at time $0$ and the diameter of $M^N(0)$ is $1 > \frac34 = \Delta$. But $X^3_2(k) \in (\frac14,\frac34)$ for all $k$ and there is always one cluster only.

\section{Mean-field limit results when $N$ goes to infinity}
\label{sec-convergence}

This section is devoted to a rigorous statement for
the following heuristic statistical mechanics limit picture: all peers act independently,
as if each were influenced by an infinite
supply of independent statistically similar peers, of which the instantaneous laws solve a
nonlinear equation
obtained by consistency from this feedback.

In statistical mechanics and probability theory,
such convergence to an i.i.d. system is called \emph{chaoticity},
and the fact that chaoticity at time $0$ implies chaoticity at further times is
called \emph{propagation of chaos}.

Some other probabilistic definitions and facts are recalled in \ref{sec-probab}.

We introduce an intermediate auxiliary system, in  which the peer meet at the instants of a Poisson process, which itself constitutes a relevant opinion model.
We adapt results in
Graham-M\'el\'eard\cite{graham1994chaos,graham1997stochastic}
to prove that the sample paths of the auxiliary system are well approximated by the limit system.
We eventually control the distance between the auxiliary system and the model of
Deffuant \emph{et al}.\cite{deffuant2000mixing}

\subsection{Mean-field regime, rescaled and auxiliary systems}

The number $N$ of peers is typically large,
and we let it go to infinity. At each time-step two peers are possibly updated,
and the empirical measures have jumps of order $\frac1N$,
hence time must be rescaled by a factor $N$.
This is a mean-field limit, in which time is usually rescaled
by physical considerations (such as ``the peers meet in proportion to their numbers'').
It is also related to fluid limits.

For a Polish space $\calS$, let $\mathcal{P}(\calS)$ denote the space of probability measures
on $\calS$, with the Borel $\sigma$-field, and
$D(\Reals_+,\calS)$ denote the Skorohod space of paths
from $\Reals_+$ to $\calS$ which are
right-continuous with left-hand limits, with the product $\sigma$-field.

A non-trivial continuous-time limit process is expected for
the \emph{rescaled} system
\begin{equation}
\label{eq-resc}
\widetilde{X}^N = (\widetilde{X}^N_i)_{1\le i\le N}\,,
\qquad
\widetilde{X}^N = (\widetilde{X}^N(t), {t\in\Reals_+}) = (X^N(\lfloor Nt \rfloor), {t\in\Reals_+})\,,
\end{equation}
with
sample paths in  $D(\Reals_+,[0,1]^N)$.
The corresponding empirical measure $\widetilde{\Lambda}^N$ and
the process $\widetilde{M}^N = (\widetilde{M}^N(t), {t\in\Reals_+})$
constituted of its marginal laws are given by
\begin{equation}
\label{eq-empresc}
\widetilde{\Lambda}^N = \frac1N \sum_{i=1}^N\delta_{\widetilde{X}^N_i}
= \frac1N \sum_{i=1}^N\delta_{(\widetilde{X}^N_i(t), {t\in\Reals_+})}\,,
\qquad
\widetilde{M}^N(t) = \frac1N \sum_{i=1}^N\delta_{\widetilde{X}^N_i(t)}\,,
\end{equation}
respectively with samples in $\calP(D(\Reals_+,[0,1]))$
and sample paths in $D(\Reals_+,\calP[0,1])$.

An \emph{auxiliary} (rescaled) system is obtained by randomizing the
jump instants of the original model by waiting i.i.d.\ exponential r.v. of mean $\frac1N$
between selections, instead of deterministic $\frac1N$ durations.
A convenient construction using a Poisson process
$(A(t), {t\in\Reals_+})$ of intensity~$1$ is that, with sample spaces as above,
\begin{equation}
\label{eq-aux}
\begin{aligned}
&\widehat{X}^N = (\widehat{X}^N_i)_{1\le i\le N}\,,
&&\widehat{X}^N = (\widehat{X}^N(t), {t\in\Reals_+})
= (X^N(A(Nt)), {t\in\Reals_+})\,,
\\
&\widehat{\Lambda}^N = \frac1N \sum_{i=1}^N \delta_{\widehat{X}^N_i}\,,
&&\widehat{M}^N = (\widehat{M}^N(t),{t\in\Reals_+})\,,\;\;
\widehat{M}^N(t) = \frac1N \sum_{i=1}^N \delta_{\widehat{X}^N_i(t)}\,.
\end{aligned}
\end{equation}

If $T_k$ for $k\ge0$ are given by $T_0=0$ and the jump instants of $(A(t), {t\in\Reals_+})$, then
\begin{equation}
\label{eq-rescaux}
\quad
\widetilde{X}^N(t) = \widehat{X}^N(t') = X^N(k)\,,
\qquad
\frac{k}N\le t < \frac{k+1}N\,,
\quad
\frac{T_k}N\le t' < \frac{T_{k+1}}N\,.
\end{equation}
Note that $\widetilde{M}^N(t) = M^N(\lfloor Nt \rfloor)$
and $\widehat{M}^N(t) = M^N(A(Nt))$,
but that the relationship between
$\widetilde{\Lambda}^N $ and $\widehat{\Lambda}^N $ and $\Lambda^N$
is more involved.

The process $\widehat{X}^N$ is a pure-jump Markov process with rate bounded by $N$, at which
two peers are chosen uniformly at random without replacement, say $i$ and $j$
at time $t$, and:
\begin{itemize}
\item
if $|\widehat{X}^N_i(t-)- \widehat{X}^N_j(t-)| >\Delta$ then $\widehat{X}^N(t) = \widehat{X}^N(t-)$,
\item
if $|\widehat{X}^N_i(t-)-\widehat{X}^N_j(t-)| \le \Delta$ then only the values of peers $i$ and $j$
change to
\[
\left\{
\begin{aligned}
\widehat{X}^N_i(t) &= w \widehat{X}^N_i(t-) + (1-w) \widehat{X}^N_j(t-)\,,
\\
\widehat{X}^N_j(t) &= w X^N_j(t-) + (1-w) \widehat{X}^N_i(t-)\,.
\end{aligned}
\right.
\]
\end{itemize}

\begin{remark}
\label{r-ratecomp}
Each of the $\frac{N(N-1)}2$ unordered pairs of peers is thus
chosen at rate $\frac2{N-1} = N / \frac{N(N-1)}2$, and then both peers
undergo a \emph{simultaneous} jump in their values if these are close enough.
Each peer is thus affected at rate $2 = (N-1)\frac2{N-1}$.
\end{remark}

The generator $\calA^N$ of $\widehat{X}^N = (\widehat{X}^N_n)_{1\le n\le N}$
acts on $f \in L^\infty([0,1]^N)$ (the Banach space of essentially bounded measurable functions
on $[0,1]^N$) as
\begin{equation}
\label{eq-gensys}
\calA^N f ((x_n)_{1\le n\le N})
= \frac{2}{N-1} \sum_{1\le i < j\le N}
[f((x_n)_{1\le n\le N}^{i,j}) - f((x_n)_{1\le n\le N})] \ind{|x_i-x_j|\le\Delta}
\end{equation}
where $(x_n)_{1\le n\le N}^{i,j}$ is obtained from $(x_n)_{1\le n\le N}$ by replacing $x_i$ and $x_j$ with
$w x_i + \allowbreak(1-w) x_j$ and $w x_j + (1-w) x_i$  and leaving the other coordinates fixed.
Its operator norm is bounded by $2N$, and the law of the corresponding Markov process $\widehat{X}^N$
is well-defined in terms of the law of $\widehat{X}^N(0) = X^N(0)$.

For $1\le k \le N$, this generator
acts on $h_k \in L^\infty([0,1]^N)$ which depend only on the $k$-th coordinate,
of the form $h_k((x_n)_{1\le n\le N}) = h(x_k)$ for
some $h\in L^\infty[0,1]$, as
\begin{multline}
\label{eq-margen}
\frac{2}{N-1} \sum_{\substack{1\le j\le N : j\neq k}}
[h(w x_k + (1-w) x_j)-h(x_k)]  \ind{|x_k-x_j|\le\Delta}
\\
:=
\calA\Biggl( \frac{1}{N-1} \sum_{\substack{1\le j\le N : j\neq k}} \delta_{x_j}(dy)\Biggr)h(x_k)
\end{multline}
where the generators $\calA(\mu)$
act on $h\in L^\infty[0,1]$ as
\begin{equation}
\label{eq-genfam}
\calA(\mu) h (x)
=
2\langle [h(w x + (1-w) y)-h(x)] \ind{|x-y|\le\Delta} \,, \mu(dy) \rangle\,,
\quad
\mu \in \calP[0,1]\,.
\end{equation}

Heuristically, if the
$\widehat{X}^N_i(0)$ converge in law to i.i.d. r.v. of law $m_0$, then the $\widehat{X}^N_i$ are expected to converge in law to i.i.d.
processes of law $Q$, the law of a
time-inhomogeneous
Markov process  $\widehat{X}$ with initial law $m_0$ and generator $\calA (m(t))$ at time $t\in\Reals_+$,
where $m(t) = \calL(\widehat{X}_t) = Q_t$ is the instantaneous law of this same process,
the marginal of $Q$.
Such a process is called a nonlinear Markov process, or a McKean-Vlasov process.
Considering the forward Kolmogorov equation for this Markov process,
 $(m(t), {t\in\Reals_+})$ should satisfy the following weak (or distributional-sense)
 formulation of a nonlinear integro-differential
equation.

\begin{definition}[Problem 1]
\label{d-prob1}
We say that $m=(m(t), {t\in\Reals_+})$ with $m(t) \in \calP[0,1]$
is solution to Problem~1 with initial value
$m_0 \in \calP[0,1]$ if $m(0) = m_0$ and
\begin{align}
\label{eq-fm}
&\langle h, m(t) \rangle
- \langle h, m(0) \rangle
= \int_0^t \langle \calA (m(s))h, m(s)\rangle\,ds
\notag\\
&\qquad :=
\int_0^t 2 \big\langle [h(w x + (1-w) y)-h(x)] \ind{|x-y|\le\Delta} \,,
m(s)(dy) m(s)(dx)
\big\rangle\,ds
\end{align}
for all test functions $h\in L^\infty[0,1]$;
this can be written more symmetrically as
\begin{align}
\label{eq-alt-def-pb1}
\langle h, m(t) \rangle
- \langle h, m(0) \rangle
&=
\int_0^t
\big\langle [h(w x + (1-w) y) + h(w y + (1-w) x)
\notag\\
&\qquad
 - h(x) -h(y)] \ind{|x-y|\le\Delta} \,,
m(s)(dy) m(s)(dx)
\big\rangle\,ds\,.
\end{align}
\end{definition}

The distance in total variation norm of $\mu$ and $\mu'$ in $\calP(\calS)$ is
given by
\begin{equation}
\label{tvn}
|\mu - \mu'| = \sup_{\Vert\phi \Vert_\infty\le 1} \langle \phi, \mu-\mu' \rangle
= 2\sup\{\mu(A)-\mu'(A) :  \text{measurable } A\subset \calS \}\,.
\end{equation}

\begin{theorem}
\label{theo-eunl}
Consider the generators $\calA(\mu)$ given by \eqref{eq-genfam}, and  $m_0$ in $\calP[0,1]$.
\begin{enumerate}
\item
\label{eunl1}
There is a unique solution $m=(m(t), {t\in\Reals_+})$ to Problem 1 starting at $m_0$.
For the total variation norm on $\calP[0,1]$,
$t \mapsto m(t)$ is continuous, and
$m_0\mapsto (m(t), {t\in\Reals_+})$ is continuous for uniform convergence on bounded time sets.

\item
\label{eunl2}
There is a unique law $Q=\calL(\widehat{X})$ on $D(\Reals_+,[0,1])$
for an inhomogeneous Markov process $\widehat{X} = (\widehat{X}(t), t\in \Reals_+)$ with
generator $\calA (m(t))$ at time $t$ and initial law $\calL(\widehat{X}(0))= m_0$.
Its marginal $Q_t=\calL(\widehat{X}_t)$ is given by $m(t)$.
\end{enumerate}
\end{theorem}

\begin{remark}
\label{r-weak-funct}
Such nonlinear Markov processes and equations
are well-known to probabilists.
The equations 1 have same probabilistic structure as the weak forms (2.1), (2.2), (2.4)
(with $\mathcal{L}=0$)
of the spatially homogeneous version (without $x$-dependence) of the Boltzmann equation (1.1) in
Graham-M\'el\'eard\cite{graham1997stochastic},
the weak form (1.7) of the (cutoff) Kac equation (1.1)-(1.2)
in Desvillettes \emph{et al}.,\cite{desvillettes1999probabilistic}
 the nonlinear Kolmogorov equation (2.7) in Graham,\cite{graham2000chaoticity}
and the kinetic equation (9.4.4) in Graham.\cite{graham2000kinetic}
The weak formulation involves explicitly the generator of the underlying Markovian dynamics and allows to understand it more directly. The
functional formulation (for probability density functions) of this integro-differential equation involves an adjoint expression of this generator, and will be seen in Section~\ref{sec-num}.
\end{remark}

\subsection{Difficulties for classical mean-field limit proofs}
\label{s-diffmfpr}

The system $\widehat{X}^N$ exhibits simultaneous jumps in two coordinates, and
is in \emph{binary mean-field interaction} in
statistical mechanics terminology.

A system in which only one opinion would change at a time would be
in \emph{simple} mean-field interaction; the generator $\calA^N$ in \eqref{eq-gensys}
would be replaced by a simpler expression, which could be written
 as a sum over $i$ of terms acting only on the $i$-th coordinate in terms of the value $x_i$ and of $\frac1{N-1}\sum_{j\neq i}\delta_{x_j}$. Consequently, the empirical measures
would satisfy an equation in almost closed form, which could be exploited in various ways to prove convergence to a limit
satisfying the closed
nonlinear equation in which the empirical distribution is replaced by the law itself.

A binary mean-field interacting system is much more complex,
since there is much more feedback between peers.
It is
impossible to relate it in a simple way to an independent system, in which the coordinates cannot jump simultaneously.
Because of that, the coupling methods introduced by
Sznitman,\cite{sznitman1991topics} see also M\'el\'eard\cite{meleard1996asymptotic} and Graham-Robert,\cite{graham2009interacting} cannot be adapted here.
Moreover, these use contraction techniques, and the metric used is too weak for
the indicator functions.

Elaborate compactness-uniqueness methods are also used for proofs,
see Sznitman,\cite{sznitman1991topics} and also M\'el\'eard,\cite{meleard1996asymptotic}
Graham-M\'el\'eard\cite{graham1997stochastic} Section~4,
and Graham,\cite{graham2000chaoticity,graham2000kinetic} but require weak topologies for compactness criteria, and continuity properties
in order to pass to the limit; hence, the indicator functions prevent
using them here.

\begin{remark}
The indicator functions require quite strong topologies.
For instance, if $0<a < b = a+\Delta<1$
and $m_0= \frac12(\delta_{a} + \delta_{b})$, then
there exists $M^N_+(0)$ with support not intersecting $[a,b]$ and converging weakly to
$m(0)$, and starting there $M^N_+(k)$ and $m^N_+(t)$ have at least two clusters and support outside $[a, b]$. There exists also $M^N_-(0)$ with support inside $(a,b)$
and converging weakly to
$m(0)$, and $M^N_-(k)$ and $m^N_+(t)$ have one cluster and support inside $(a, b)$
for any $k\in\Nats$, and will be a total consensus after some random time.
\end{remark}

\subsection{Rigorous mean-field limit results for the auxiliary system}

Systems of this type were studied in
Graham-M\'el\'eard,\cite{graham1994chaos,graham1997stochastic}
see also Ref.~\refcite{desvillettes1999probabilistic}.
The first paper studied a class of  not necessarily Markovian multitype
interacting systems, as a model for
communication networks. The second studied Monte-Carlo methods for a class of Boltzmann models,  and in particular expressed  some notions and results of
the first in this framework. Their results yield the following.

For $k\ge1$ and $T\ge0$ and laws $P$ and $P'$ on $D(\Reals_+,[0,1]^k)$, let
$|P-P'|_T$ denote the distance in variation norm \eqref{tvn} of
the restrictions of $P$ and $Q$ on $D([0,T],[0,1]^k)$.
When clear, the processes will be restricted to $[0,T]$ without further mention.

\begin{theorem}
\label{theo-chaotic-aux}
Consider the auxiliary system \eqref{eq-aux} for $N\ge2$.
If the $\widehat{X}^N_i(0):=X^N_i(0)$ are i.i.d. of law $m_0$, then
there is propagation of chaos.
More precisely, let $m=(m(t), {t\in\Reals_+})$ and $Q$ be as in \thref{theo-eunl} for $m(0)=m_0$,
and $T>0$.
\begin{enumerate}
\item
For $1\le k \le N$,  \label{chaotic-aux1}
\[
| \calL(\widehat{X}^N_1, \dots, \widehat{X}^N_k)
- \calL(\widehat{X}^N_1)\otimes \dots \otimes \calL(\widehat{X}^N_k)|_T
\le 2k(k-1)\frac{2T + 4 T^2}{N-1}\,,
\]
and
\[
\Biggl|\frac1N \sum_{i=1}^N \calL(\widehat{X}^N_i) - Q\Biggr|_T
\le
| \calL(\widehat{X}^N_i)
- Q|_T
\le 6 \frac{\exp(2T)-1}{N+1}\,.
\]

\item
\label{chaotic-aux2}
For any $\phi:D([0,T],[0,1])\to\Reals$ such that $\Vert\phi\Vert_\infty \le 1$,
\[
\E\Biggl[
\biggl\langle\phi ,
\widehat{\Lambda}^N
- \frac1N \sum_{i=1}^N \calL(\widehat{X}^N_i)
\biggr\rangle^2
\Biggr]
\le \frac{4+8T + 16T^2}N\,.
\]
Moreover
\[
\widehat{\Lambda}^N \xrightarrow[N\to \infty]{\textnormal{in probab.}} Q\,,
\qquad
\widehat{M}^N \xrightarrow[N\to \infty]{\textnormal{in probab.}} m\,,
\]
respectively for the weak topology on
$\calP(D(\Reals_+,[0,1]))$ with the
Skorohod topology on $D(\Reals_+,[0,1])$,
and for the topology of uniform convergence on bounded time intervals on
$D(\Reals_+,\calP[0,1])$
with the weak topology on $\calP[0,1]$.
\end{enumerate}
\end{theorem}

The assumption that the initial conditions are i.i.d. can be appropriately relaxed,
as in Theorem~1.4 in Graham-M\'elé\'eard\cite{graham1993propagation}.

These very strong results are obtained for a relevant opinion model
given by the auxiliary (rescaled) system, and are of independent interest.
In the next section
we will derive from them some weaker results for the original discrete-time model.

\begin{remark}
\label{r:prob-law}
The convergence result for $\widehat{\Lambda}^N$ is equivalent to convergence in
law to $Q$ (Ethier-Kurtz,\cite{ethier1986markov} Corollary~3.3.3).
The convergence result for  $\widehat{M}^N$
implies convergence in law to $m$ for test functions which are continuous, bounded, and measurable for the product $\sigma$-field
(Ref.~\refcite{ethier1986markov}, Theorem 3.10.2).
Separability issues restrict these results, see \ref{sec-probab}; in fact, convergence of
$\widehat{\Lambda}^N$ holds for
any convergence induced by a denumerable set of bounded measurable functions.
\end{remark}

\subsection{From the auxiliary to the rescaled system}

For $k\ge1$, let $a_k$ denote the Skorohod metric on $D(\Reals_+,[0,1]^k)$  given by (3.5.21)
in Ethier-Kurtz\cite{ethier1986markov} for the atomic metric
$(x,y)\mapsto\ind{x\neq y}$ on $[0,1]^k$ (which
induces the topology of all subsets of $[0,1]^k$,
for which any function is continuous). Note that $a_k$ is measurable with respect to the
usual Borel $\sigma$-field on $[0,1]^k\times [0,1]^k$.

A \emph{time-change} is an increasing homeomorphism of $\Reals_+$, \emph{i.e.}, a continuous
function from $\Reals_+$ to $\Reals_+$ which is null at the origin and strictly increasing to infinity. Two paths are close for $a_k$ if there is a time-change close to the identity such that
the time-change of one path is equal to the other path.

Eq.~\eqref{eq-rescaux} is the key to obtain the following quite general result
showing that the rescaled system $\widetilde{X}^N$ is very close to
the the auxiliary system $\widehat{X}^N$, up to a well-controlled (random)
time-change.

\begin{theorem}
\label{theo-aux-resc}
Consider the rescaled system \eqref{eq-resc} and the auxiliary system \eqref{eq-aux} for $N\ge2$.
Then $
\lim_{N\to\infty}
a_N(\widetilde{X}^N, \widehat{X}^N)=0
$
in probability.
\end{theorem}

This result and Theorem~\ref{theo-chaotic-aux} now yield the main
mean-field convergence result.

\begin{theorem}
\label{theo-chaotic-resc}
Consider the rescaled system \eqref{eq-resc} for $N\ge2$.
If the $\widetilde{X}^N_i(0):=X^N_i(0)$ are i.i.d. of law $m_0$, then
there is propagation of chaos.
More precisely, let $m=(m(t), {t\in\Reals_+})$ and $Q$ be as in \thref{theo-eunl} for $m(0)=m_0$.
\begin{enumerate}
\item
\label{chaotic-resc1}
For $1\le k \le N$,
\[
\lim_{N\to\infty} \calL(\widetilde{X}^N_1, \dots, \widetilde{X}^N_k) = Q^{\otimes k}\,,
\]
for the weak topology on $\calP(D(\Reals_+,[0,1]^k))$
induced by test functions which are either uniformly continuous
for the Skorohod metric $a_k$, bounded, and measurable for the usual product $\sigma$-field
(for the usual Borel $\sigma$-field on $[0,1]^k$),
or continuous for the usual Skorohod topology (for the usual metric on $[0,1]^k$) and bounded.

\item
\label{chaotic-resc2}
For the usual topology of $[0,1]$,
\[
\widetilde{\Lambda}^N \xrightarrow[N\to \infty]{\textnormal{in probab.}}  Q\,,
\qquad
\widetilde{M}^N \xrightarrow[N\to \infty]{\textnormal{in probab.}}  m\,,
\]
respectively for the weak topology on
$\calP(D(\Reals_+,[0,1]))$ with the
Skorohod topology on $D(\Reals_+,[0,1])$,
and for the topology of uniform convergence on bounded time intervals on
$D(\Reals_+,\calP[0,1])$
with the weak topology on $\calP[0,1]$.
\end{enumerate}
\end{theorem}

The assumption that the initial conditions are i.i.d. may again be relaxed.
For the second result, see again Remark~\ref{r:prob-law}.

\section{Infinite $N$ Model}

\label{sec-infinite}

We now study the mean-field
limit $m=(m(t),{t\in\Reals_+})$ obtained in \sref{sec-infinite} when $N$ goes to infinity.
As for the finite $N$ model, we find that there is convergence
to a partial consensus as $t$ goes to infinity. The limit may
depend on the initial conditions, as might the random limit when $N$ is
finite. 
We are able to say more, and notably find
tractable sufficient conditions for the limit to be a total
consensus.

\subsection{Convexity and Moments}
Applying \pref{lem-convex} to the equivalent definition of
Problem 1 given by \eqref{eq-alt-def-pb1} yields the following:

\begin{corollary}
Let $m=(m(t),{t\in\Reals_+})$ be a solution of
Problem~1. If $h:[0,1] \to \Reals$ is convex, then $\cro{h,m(t)}$ is
a non-increasing function of $t$.\label{coro-cv-fl}
Moreover, 
for $n =1,2,\dots$ and $t\in\Reals_+$, let $\mu_n(t) =
\int_0^1 x^n \; m(t)(dx)$ denote the $n$-th moment of $m(t)$, and $\sigma(t)$
its standard deviation (\emph{i.e.}, $\sigma(t)^2 = \mu_2(t)-\mu_1(t)^2$).
Then:
\begin{enumerate}
  \item  The mean $\mu_1(t)$ is stationary: $\mu_1(t)=\mu_1(0)$ for all $t$.
  \item The moments $\mu_n(t)$ are non-increasing in $t$:
  if $t_1\le t_2$ then $\mu_n(t_1) \ge \mu_n(t_2)$.

  \item The standard deviation  $\sigma(t)$ is also a non-increasing function of $t$.
\end{enumerate}%\label{discrete_moments}
\end{corollary}%
%
% Let $\sigma(t)$ be the standard deviation of $m(t)$, \emph{i.e.},
% $\sigma(t)^2= \mu_2(t)-\lp\mu_1(t)\rp^2 $. It follows immediately
% that
%\begin{corollary}
%%\label{discrete_variance_drop}
%$\sigma(t)$ is a
%non-increasing function of $t$. %Moreover, if peers $i$ and $j$
%interact and $x_i(k)$ and $x_j(k)$ pass the deviation test, the
%drop on the variance is given by
%$\frac{1}{N^2}2w(1-w)(x_i(k)-x_j(k))^2$. Otherwise the variance
%is the same.
%\end{corollary}
Furthermore, we have some bounds.

 \begin{proposition}
 For all $t \geq 0$, we have
 $\sigma(0) \geq \sigma(t) \geq %\color{red}
 \sigma(0) e^{-4w(1-w)t}
 %\color{black}
 $.
  \label{pro-stdevbd}
 \end{proposition}

%\color{red}
Note that \coref{coro-cv-fl} and \pref{pro-stdevbd} generalize
results of Ref~\refcite{ben2003bifurcations}, which established
similar results for the case $w=1/2$. However, the bound in
\pref{pro-stdevbd} is different, as the equation considered in
Ref. \refcite{ben2003bifurcations} misses a factor 2.
%\color{black}

\subsection{Convergence to Partial Consensus}
It is immediate that a partial consensus is a stationary point for Problem~1,
\emph{i.e.}, if $(m(t),{t\in\Reals_+})$ is solution of Problem~1 with initial value
a partial consensus $m_0$, then  $m(t)=m_0$ for all $t$.
Conversely, we
show, in Theorem~\ref{th-mfcvpc} below, that any trajectory
$(m(t),{t\in\Reals_+})$ converges to a partial consensus.

It is useful to consider the essential sup and
inf of $m(t)$, defined as follows.

\begin{definition}
 For $\nu \in \calP[0,1]$,
 let $\textnormal{ess\,sup}(\nu) = \inf\{b \in [0,1], \nu \lp (b,1]\rp=0\}$
 and
 $\textnormal{ess\,inf}(\nu) = \sup\{a \in [0,1], \nu\lp [0,a)\rp=0\}$.
 \end{definition}

Note that if $\textnormal{ess\,inf}(\nu)=a$ and
$\textnormal{ess\,sup}(\nu)=b$, then the support of $\nu$
is included in $[a,b]$, \emph{i.e.}, for any measurable
$B \subset [0,1]$, $\nu(B)=\nu\lp B \cap
[a,b]\rp$.%, and $[a,b]$ is the convex envelope of
%the support of $\nu$.

\begin{proposition}
Let $(m(t),{t\in\Reals_+})$ be solution of
Problem~1. Then $\textnormal{ess\,sup}(m(t))$ [resp.
$\textnormal{ess\,inf}(m(t))$] is a non-increasing [resp.
non-decreasing] function of $t$.
\label{pro-supp}
\end{proposition}

%[JYLB: Below is the weak form of the theorem that I could
%prove. I am not convinced that it is possible to say more.
%
%I propose to stay with this weak form, at least provisionally.
%
%The following corollaries and theorems may need to be modified
%accordingly.]

See Definition~\ref{d-partcons} for partial and total consensus.

\begin{theorem}
\label{th-mfcvpc}
Let $(m(t),{t\in\Reals_+})$ be a solution of Problem~1. As $t$
goes to infinity, $m(t)$ converges, for the weak topology on
$\calP[0,1]$, to some $m(\infty)$ which is a partial consensus for
every $\Delta' < \Delta$, \emph{i.e.}, of the form
$m(\infty) = \sum_{m=1}^{c}\alpha_m \delta_{x_m}$ with
 $x_m \in [0,1]$, $\abs{x_m-x_{m'}}
\geq \Delta$ for $m \neq m'$, and $\alpha_m >0$.
\label{theo-cv-pc}
\end{theorem}

Note that the limit $m(\infty)$ may depend on the
initial condition $m_0$, and may or may not be a
total consensus (as shown in the next section).
We are in particular interested in finding
initial conditions that guarantee that
$m(\infty)$ is a total consensus. The following
is an immediate consequence of \pref{pro-supp}.

\begin{corollary}
If the diameter of $m_0$ is  less than
$\Delta$, \emph{i.e.}, if
$\textnormal{ess\,sup}(m_0)-\textnormal{ess\,inf}(m_0) <
\Delta$, then $m(\infty)$ is a total consensus.
\end{corollary}
Note that the converse is not true: if the
diameter of $m_0$ is larger or equal than $\Delta$, there
may be convergence to total consensus (see next
section for an example).

\subsection{Convergence to Total Consensus}

We find sufficient criteria for guaranteeing
some upper bounds on the number of components of $m(\infty)$,
in particular, we find some sufficient conditions for
convergence to total consensus. Although the bounds
are suboptimal, to the best of our knowledge, they are the
first of their kind. The bounds are based on
\coref{coro-cv-fl}.

%First define, for $n \in \{2,3,\dots\}$ and $\eps
%>0$, the set
% \ben E_{n,\eps}=\{x=(x_1,\dots,x_n) \in [0,1]^n,
%  x_{i+1} \geq x_{i} + \Delta + \eps \; \forall i=1,\dots,n-1\}
%  \een
%Note that if $x \in E_{n,\eps}$, then
%$\frac{1}{n}\sum_{i=1}^n \delta_{x_i}$ is a
%partial consensus with $n$ components, and
%vice-versa, any partial consensus with $n$
%components can be put in this form for some $\eps
%>0$.
First define, for $n \in \{1,2,3, \dots\}$ and $\mu_0
\in [0,1]$,  the set $P_{n}(\mu_0)$ of
partial consensus with $n$ components and mean
$\mu_0$, \emph{i.e.}, $\nu \in P_{n}(\mu)$ iff there is
some sequence $0\leq x_1<\dots<x_{n}\leq 1$ with
$x_i + \Delta < x_{i+ 1}$, some sequence
$\alpha_i$ for $ i=1,\dots,n$ with $0<\alpha_i<1$ and
$\sum_{i=1}^{n}\alpha_i=1$ such that
$\nu=\frac{1}{n}\sum_{i=1}^{n} \alpha_i
\delta_{x_i}$ and $\frac{1}{n}\sum_{i=1}^{n}
\alpha_i x_i = \mu_0$.
%
% \ben E_{n}(\mu)=\lc x=(x_1,\dots,x_n) \in [0,1]^n,
%  x_{i+1} > x_{i} + \Delta  \; \forall i=1,\dots,n-1  \; \mand \frac{1}{n}\sum_{i=1}^n x_i =
%  \mu\rc
%  \een
%Note that if $x \in E_{n}$, then
%$\frac{1}{n}\sum_{i=1}^n \delta_{x_i}$ is a
%partial consensus with $n$ components and mean
%$\mu$, and vice-versa, any such partial consensus can be put in this form.

Second, for any convex, continuous $h: [0,1] \to
 \Reals_+$, let $Q_{n}(\mu_0, h)$ be
 the set of strict lower bounds of the image by the mapping $\nu \mapsto
 \cro{h,\nu}$ of $P_{n}(\mu_0)$, \emph{i.e.}, $q \in Q_{n}(\mu_0,h)$ iff
 for any
 consensus $\nu$ with ${n}$ components and mean
 $\mu_0$, it holds that
 $\cro{h,\nu} > q$. If $P_{n}(\mu_0)$ is empty, let $Q_{n}(\mu_0,
 h)=\Reals_+$.
% \color{red}
% [(Javi): I think if we take strictly convex functions instead of ''plain'' convex we can take $Q_n$ as the lower bounds (not strict) and all the %theorems and corollaries follow. Otherwise the bounds for $\Delta$ have to be strictly inferior. What do you think?]
% \color{black}

 Note that  $Q_{n}(\mu_0, h)$ is necessarily an
 interval with lower bound $0$.
The following proposition states that $Q_n$ is
non decreasing with $n$.

\begin{proposition} 
For any $n \in \{1,2,3, \dots\}$ and $\mu_0\in [0,1]$ and convex continuous 
$h: [0,1] \to\Reals_+$, it holds that
$Q_n(\mu_0, h) \subset Q_{n+1}(\mu_0, h)$.
 \label{prop-qdec}
 \end{proposition}

 Combining \pref{prop-qdec} with
 \coref{coro-cv-fl}, we obtain:
% \ben
% Q_{m_0}(\mu_0, h) = \lc
% q \in \Reals_+, \forall \nu \in P_{m_0}(\mu_0): \; \frac{1}{{m_0}}\sum_{i=1}^{m_0} h(x_i) >q
% \rc
% \een \emph{i.e.}, $q \in Q_{m_0}(\mu_0,h)$ iff for any
% consensus $\nu$ with ${m_0}$ components and mean $\mu_0$,
% $\cro{h,\nu}>q$. If $P_{m_0}(\mu_0)$ is empty, then  $Q_{m_0}(\mu_0, h)=\Reals_+$.
%
%\begin{definition}
% For a convex, continuous $h: [0,1] \to
% \Reals_+$, $ n \in \{2,3...\}$ and $\eps >0$
%let
% $Q_{n,\eps}(h) = \frac{1}{n}\min_{x \in
%    E_{n,\eps}}\sum_{i=1}^n h(x_i)
%$ and  $Q_n(h)=\inf_{\eps > 0} Q_{n,\eps}(h)$.
%\end{definition}
%Note that $E_{n,\eps}$ is compact therefore the
%minimum in $Q_{n,\eps}(h)$ is well defined.

\begin{theorem}
Let $(m(t),t\in \Reals_+)$ be the solution of Problem~1 with initial
condition $m_0$, and $c$ be the number of
components of the limiting partial consensus
$m(\infty)$. Assume that, for some $n \in
\{2,3,\dots\}$, some convex continuous $h: [0,1] \to
 \Reals_+$, and some $q\geq 0$, we have $q \in  Q_n(\mu_0,
 h)$, where $\mu_0$ is the mean of $m_0$.
 
Under these assumptions, if $\cro{h,m_0} \leq q$ then $c \leq n-1$.
 \label{theo-ccbq}
\end{theorem}

Here is an example of use of the theorem, for $n=2$ and $h(x) = \abs{x-\mu_0}$.

\begin{corollary} Let $(m(t),{t\in\Reals_+})$ be the solution of Problem~1 
starting at $m_0$. Assume that $\Delta \geq \frac12$ and
$1-\Delta\leq \mu_0\leq  \Delta$,
where $\mu_0$ is the mean of $m_0$.
 If \ben
 \int_0^1 \abs{x-\mu_0} m_0(dx) <
 \frac{2}{\Delta}\min
   \lc
      \mu_0(\Delta-\mu_0), \;  (1-\mu_0)(\Delta-1+\mu_0)
   \rc
\een
 then $m(t)$ converges to total
 consensus.
 \end{corollary}
If we apply this to $m_0$ equal to the uniform
distribution, we find the sufficient condition
$\Delta > \frac23$ for convergence to total
consensus. In \coref{co-unif} we find a better
result, obtained by exploiting symmetry
properties.

 \begin{definition}
 We say that $\nu\in
\calP[0,1]$ is symmetric if the image measure of $\nu$ by
$x\mapsto 1-x$ is $\nu$ itself.
\end{definition}

Note that if
$\nu$ has a density $f$, this simply means that
$f(x)=f(1-x)$. Necessarily, if $\nu$ is
symmetric, the mean of $\nu$ is $\frac12$. If a
partial consensus $\nu=\frac{1}{n}\sum_{i=1}^{n}
\alpha_i \delta_{x_i}$ (with $x_i<x_{i+1}$) is
symmetric, then $x_{n+1-i}=1-x_i$ and
$\alpha_{n+1-i}=\alpha_i$; in particular, if $n$
is odd, $x_{\frac{n+1}{2}}=\frac12$.

 \begin{proposition}
 Let $(m(t),{t\in\Reals_+})$ be  a solution of Problem~1 with initial value $m_0$.
 If $m_0$ is symmetric, then $m(t)$ is symmetric for
 all $t\geq 0$.
 \label{prop-invalsym}
 \end{proposition}

We can extend the previous method to the
symmetric case as follows. Define $SP_{n}$ as
the set of \emph{symmetric} partial consensus
with $n$ components, and let $q \in SQ_{n}(h)$ if and only if
every symmetric consensus $\nu$ with ${n}$ components satisfies
 $\cro{h,\nu} > q$. If $SP_{n}$ is empty, then
 $SQ_{n}(h)=\Reals_+$. We have similarly:

 \begin{proposition} For any $n \in \{1,2,3,\dots\}$ and convex continuous 
 $h: [0,1] \to \Reals_+$,
 it holds that
 $
 SQ_n(h) \subset SQ_{n+1}(h)
 $.
 \label{prop-qdec-s}
 \end{proposition}

%
% \begin{proposition} For $n \in \{1,2,3,\dots\}$, and any convex, continuous $h: [0,1]
%\to
% \Reals_+$, we have
% $
% SQ_n(h) \subset SQ_{n+2}(h)
% $ and $
% SQ_{2n-1}(h) \subset SQ_{2n}(h)
% $.
% \label{prop-qdec-s}
% \end{proposition}

\begin{theorem}
Let $(m(t),{t\in\Reals_+})$ be the solution of Problem~1 for a
\emph{symmetric} initial condition $m_0$, and
$c$ be the number of components of the limiting
partial consensus $m(\infty)$. Assume that, for
some $n \in \{2,3,\dots\}$, some convex continuous
$h: [0,1] \to \Reals_+$, and some $q\geq 0$, we have $q \in  SQ_n(h)$.

Under these assumptions, if $\cro{h,m_0} \leq  q$ then $c \leq n-1$.

 %\begin{itemize}
%    \item $d\leq n-1$ or $d=n+1$ if $n$ is
%    even,
%    \item $d \leq n-1$ if $n$ is odd.
% \end{itemize}
\label{theo-ccbq-s}
\end{theorem}

We apply \thref{theo-ccbq-s} with
$h(x)=\abs{x-\frac12}$. It is easy to see that for
$\nu \in SP_2$ we have  $\cro{h,\nu} \geq
\frac{\Delta}{2}$, which shows the following:

\begin{corollary} Let $(m(t),{t\in\Reals_+})$ be the solution of Problem~1 for 
a \emph{symmetric} initial
condition $m_0$. If %$\Delta \geq 1-4 \int_0^{\frac12}x\;m_0(dx)$
$\Delta > 2\int_0^{1}\abs{x-\frac12}\;m_0(dx)$ then $m(t)$ converges
either to a total consensus or to a partial consensus with 3 or more components.
 \end{corollary}

 %[(Javi): To be added: Optimality of $\abs{x - \frac12}$ as kernel for the symmetric case]

 In particular, if $m_0$ is the uniform
 distribution on $[0,1]$, then $\int_0^{1} \abs{x - \frac12} m_0(dx)=\frac{1}{4}$
 and the condition in the previous corollary is $\Delta
 > \frac12$,
 %. In this case any partial consensus
 %has at most 2 components,
 thus we have shown:

\begin{corollary} Let $(m(t),{t\in\Reals_+})$ be the solution of Problem~1 with initial
condition the uniform distribution on $[0,1]$. If
$\Delta > \frac12$ then $m(t)$ converges to a total
 consensus.\label{co-unif}
 \end{corollary}

 \begin{corollary}
 \label{coro-cie}
 Let $(m(t),{t\in\Reals_+})$ be a solution of Problem~1 with initial
condition $m_0 =
\left(\frac{1-\alpha}{2}\right)\delta_0 +
\alpha\delta_{\frac{1}{2}} +
\left(\frac{1-\alpha}{2}\right)\delta_1$. There
is convergence to total consensus 
\[
\textnormal{for }
\Delta > 1 - \alpha \textnormal{ if } \alpha \leq \tfrac{1}{2}\,,
\textnormal{ or }
\Delta > \tfrac{1}{2}  \textnormal{ if } \alpha \geq \tfrac{1}{2}\,.
\]

%\[
%\Delta \geq 1 - 2\alpha \textnormal{ if } \alpha \leq \tfrac{1}{4}\,,
%\textnormal{ or }
%\Delta \geq \tfrac{1}{2}  \textnormal{ if } \alpha \geq \tfrac{1}{4}\,.
%\]
\end{corollary}

\section{Numerical Approach}
\label{sec-num}

%%% START OF CARL

In the mean-field limit, the dynamical behavior of the system
of peers can be described by the integro-differential equation
given in weak form in Definition~\ref{d-prob1} (Problem~1).
This equation has no closed-form solution to our knowledge, and we
have developed a numerical method for it.

We describe the algorithm, and analyze its precision and
complexity. An important fact is that this algorithm requires
considerably less running time than the probabilistic methods
used in Neau\cite{neau2000revisions} when $N$ is large (which
is not surprising in dimension~$1$). The program consists in
600 lines of C++ code, and the parsing and plotting of the
results was done using Matlab.

\subsection{Functional formulation of Problem 1}

The numerical method is based on the functional formulation for
probability density functions (PDFs) obtained by duality from
the weak formulation in Definition~\ref{d-prob1} of Problem 1.
The following result is fundamental in this aspect.

\begin{theorem}
Let $(m(t),t\ge0)$ be a solution of Problem~1. If the initial
condition $m_0$ is absolutely continuous with respect to
Lebesgue measure, then so is $m(t)$ for every $t\geq 0$, and
moreover the densities $f(\cdot,t)$ of $m(t)$ satisfy the
integro-differential equation
\begin{equation}
\frac{\partial f(x,t)}{\partial t}
=
\frac{2}{w}\int_{x-\Delta w}^{x+\Delta w}f\biggl(\frac{x-(1-w)y}{w},t\biggr)f(y,t)\,dy
-2f(x,t)\int_{x-\Delta}^{x+\Delta}f(y,t)\,dy\,.
\label{eq-pdf}
\end{equation}
 Conversely, if $f: \Reals \times \Reals_+ \to \Reals$ is a
 solution of \eref{eq-pdf} such that $f(\cdot,t)$ is a PDF with support $[0,1]$
 for every $t \geq 0$, then the probability measures
 $m(t)(dx) = f(x,t)\,dx$ solve Problem~1.
 \label{theo-pdf}
\end{theorem}

This result and Theorem~\ref{theo-eunl} yield an existence and
uniqueness result for \eref{eq-pdf}. This equation can be
derived in statistical mechanics fashion by balance
considerations. For the gain term, a particle in state
$\frac{x-(1-w)y}{w}$ interacts at rate~$2$ (see
Remark~\ref{r-ratecomp}) with a particle in state $y$ to end up
in state $x$, and the joint density for this pre-interaction
configuration at time $t$ is $\frac1w
f\bigl(\frac{x-(1-w)y}{w},t\bigr)f(y,t)$ (particles are
``independent before interacting''). The loss term is derived
similarly.

\begin{remark}
As noted in Remark~\ref{r-weak-funct}, this is a Boltzmann-like
equation. This is made more obvious for $w\neq 1/2$ by the change of variables
leading to post-interaction states $x$ and $y$, which yields the equivalent formulation
\begin{multline}
\label{eq-pdf-bis} \frac{\partial f(x,t)}{\partial t} =
\frac{2}{2w-1}\int_{x-\Delta (2w-1)}^{x+\Delta (2w-1)}
f\biggl(\frac{w x - (1-w)y}{2w-1}\biggr)f\biggl(\frac{w y -
(1-w)x}{2w-1}\biggr)\,dy
\\
-2f(x,t)\int_{x-\Delta}^{x+\Delta}f(y,t)\,dy
\end{multline}
more reminiscent of Boltzmann or Kac equations such as (1.1) in
Graham-M\'el\'eard\cite{graham1997stochastic} or (1.1)-(1.2) in
Desvillettes \emph{et al}.\cite{desvillettes1999probabilistic}
In these, the fact that the gain term involves pre-collisional
velocities is obscured by the physical symmetries between
pre-collisional and post-collisional velocites, which are absent here.
\end{remark}

In the rest of this section we assume that the hypothesis of
the above theorem holds. We show next that if the PDF
$f(\cdot,0)$ is bounded then so is $f(\cdot,t)$ and we can
control its growth over time.
Let
\[M(t) \eqdef |f(\cdot,t)|_{\infty} \eqdef \sup_{x
\in [0,1]}\abs{f(x,t)}\,,
\qquad
t\ge0\,.
\]

\begin{proposition}
\label{pro-fbounded}
Then $ M(t) \leq e^{\left(\frac{2}{w}
+ \frac{2}{1-w}\right)t}(M(0) + 4) - 4$ for all $t\ge0$.
\end{proposition}

It follows that $f(\cdot,t)$ is bounded for all $t$, and
iteratively, using \eqref{eq-pdf},  $f$ is $C^{\infty}$ on its
second variable.
Having
controlled the growth of $f(x,t)$, it is easy to control the growth of its derivatives.
\begin{proposition}
\label{pro-absfbounded}
Then
$\left|\frac{\partial}{\partial
t}f(\cdot,t)\right|_{\infty} \leq \left(\frac{2}{w} +
\frac{2}{1-w}\right)(M(t) + 4)$ for all $t\ge0$.
\end{proposition}
%
%
%[JYLB: not clear to me why we need the following: ]
Iteratively, we obtain the following corollary.

\begin{corollary}
\label{cor-derivatives_bounded} If $|f(\cdot,0)|_{\infty} <
\infty$ then $\left|\frac{\partial^n }{\partial
t^n}f(\cdot,t)\right|_{\infty} < \infty$ for all $n\in\Nats$ and $t \ge0$.
\end{corollary}

\subsection{Numerical Solution of \eref{eq-pdf}}
Though Equation \eqref{eq-pdf} for the PDF does not appear to have any tractable closed form solution,
however, it lends itself well to numerical solution. We developed an algorithm that gives an approximate solution of Equation \eqref{eq-pdf} over some finite time horizon $T$, given some initial condition $f(x,0)$, assumed to be piecewise constant. %\color{red} removed text \color{black} %and given some error tolerance $\eps$ (in total variation norm).
The algorithm is described in Appendix~B. In the rest of this section, we present numerical results obtained with the algorithm.

We study
different scenarios for the initial distribution: uniform,
extremist versus undecided and beta. We find bifurcations as a function of $\Delta$. Moreover, we compare the
experimental results with the bounds obtained in section 5 and
the probabilistic Monte Carlo simulations presented in
Ref~\refcite{deffuant2000mixing}. The main results are summarized in the following table:

\begin{table}[ht]
\centering
\begin{tabular}{c|c|c}
  \hline
 \emph{ Scenario } & \emph{ Parameters } & \emph{ Consensus } \\
  \hline\hline
  Uniform &
  $\Delta \leq 0.27$, $w \in (0,1)$ & Partial
   \\
  \hline
  Uniform &
  $\Delta > 0.27$, $w \in (0,1)$ & Total
   \\
  \hline
  Extremists/Und. &
  $(\Delta,\alpha)$ below black curve (fig. \ref{fig-cie-1}) & Partial
   \\
  \hline
  Extremists/Und. &
  $(\Delta,\alpha)$ above black curve (fig. \ref{fig-cie-1}) & Total
   \\
  \hline
  Extremists/Und. &
  $(\Delta,\alpha)$ above red curve (fig. \ref{fig-cie-1}) & Partial (theor. bound)
   \\
  \hline
  Beta &
  $\Delta > 0.25$, $w \in \{0.5, 0.75\}$ & Total
   \\
  \hline
  Beta &
  $\Delta \leq 0.25$, $w \in \{0.5, 0.75\}$ & Partial
   \\
  \hline
  Beta &
  $\Delta > 0.2$, $w = 0.9$ & Total
   \\
  \hline
  Beta &
  $\Delta \leq 0.2$, $w = 0.9$ & Partial
   \\
  \hline

\end{tabular}

\caption{Summary of the numerical experiments}
\label{table-numerical-experiments}
\end{table}

\color{black}

\subsubsection{General Evolution of the System}

In order to illustrate the behavior of the system as time
passes, we show how the system evolves from a uniform
distribution to one (or more) components, depending on the
deviation threshold $\Delta$. We run those sets of experiments
for 3 different values of $w$, specifically $0.5, 0.75$, and
$0.9$ and plot the probability function at times $t = 0$, $t =
20$ and $t = 100$. The simulations have been done with the
parameters $I = 200, \Delta t = 0.1, T = 100$. Although the set
of parameters might theoretically yield a big error, in
practice this error is much smaller.

\begin{figure}[ht]
  \centering
    \includegraphics[scale=0.25]{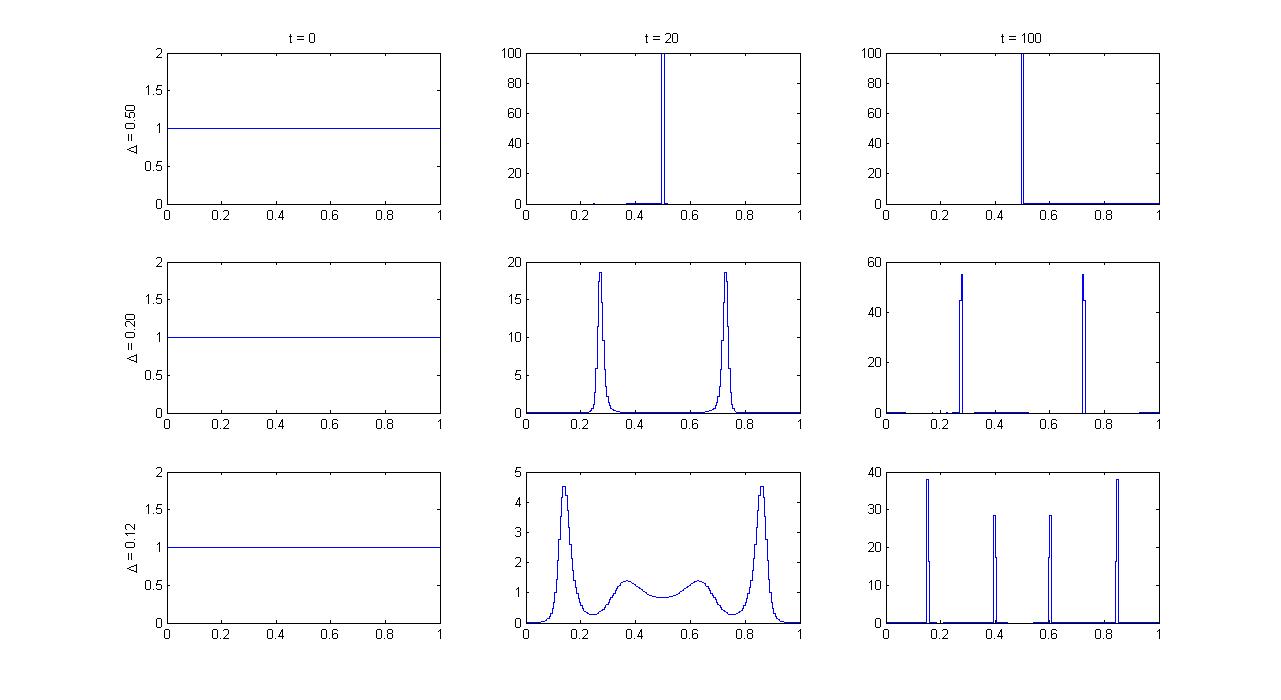}
  \caption{$w = 0.5$. Evolution  of $m(t)$ at times $t = 0, 20,100$.}
  \label{Evolution-05}
\end{figure}

\begin{figure}[ht]
  \centering
    \includegraphics[scale=0.25]{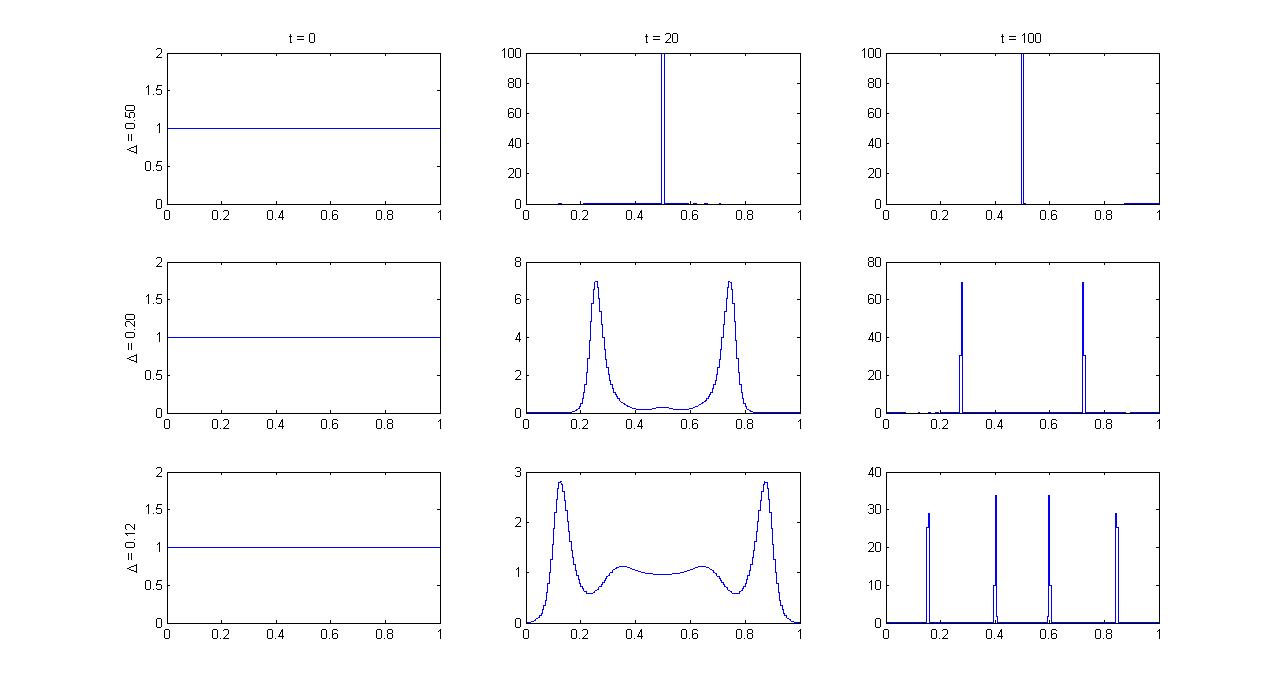}
  \caption{$w = 0.75$. Evolution of $m(t)$ at times $t = 0,20,100$.}
  \label{Evolution-075}
\end{figure}

\begin{figure}[ht]
  \centering
    \includegraphics[scale=0.25]{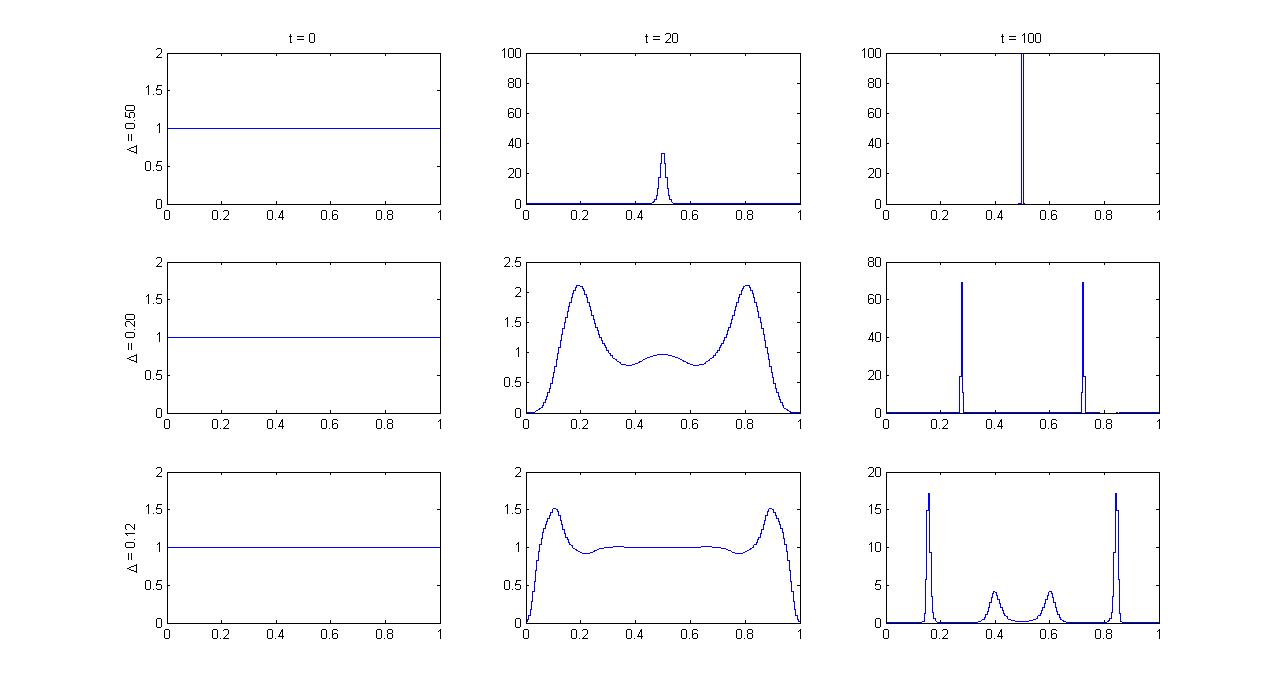}
  \caption{$w = 0.9$. Evolution of $m(t)$ at times $t = 0,20,100$.}
  \label{Evolution-09}
\end{figure}
%Simulation results for the setting in Section 6.4. Top Panels: Histogram of reputation
%ratings, with one panel per time instant, for selected times from t = 0 to t = 8192, gmany liarsh case;
%x-axis is the rating value on a 0 . 1 scale, y-axis: fraction of peers that have this value. Bottom:
%sample paths for three different honest users, starting with initial reputation rating ¸ {0, 0.5, 1}; xaxis
%is time, y-axis is rating value on a 0.1 scale. Left: N = 100 honest users; right: N = ‡. Dashed
%vertical or horizontal lines are at the true value. There are two attracting values for the ratings, the
%true value  = and a false value close to 0.

From the images, we see that $w$  does not seem
to impact the number of components of
$m(\infty)$, but the weights do depend on $w$.

\subsubsection{Extremists versus Undecided}

We now present some common scenarios: imagine a
company fusion and the opinion of the employees
about the new company, or a rough categorization
of voters in an election. We can characterize
these opinions as extremists (either 0 or 1) or
undecided (0.5). The proportion of opinions is
$\alpha$ for the undecided and
$\frac{1-\alpha}{2}$ for each of the extremist
classes. To simulate this, we have approximated the initial conditions (Diracs) to constant splines of value $I \alpha $ and $I\frac{1-\alpha}{2}$ respectively, centered at their corresponding points, such that the initial condition has mass 1.
We plot the result (1 component, i.e.
total consensus, or 2 components) for each pair
$(\alpha,\Delta)$ in $[0,1] \times
\left[\frac{1}{2},1\right]$ in \fref{fig-cie-1}.
We know from \coref{coro-cie} that total
consensus must occur for $\Delta \geq \alpha$ and
we see that the region of convergence to total
consensus is a bit larger, and slightly depends
on $w$.

Note that values of $\Delta$ smaller than $\frac{1}{2}$ would
result in no motion at all. We do this for the previous set of
values for $w$ and find that in every case, the fraction of
undecided people necessary to achieve consensus is much smaller
than what one would expect.
 %(see
%figures \ref{OneTwoDiracs-05} to \ref{OneTwoDiracs-09}).

\begin{figure}
\begin{center}
\subfigure[$w = 0.5$]{
\includegraphics[scale=0.25]{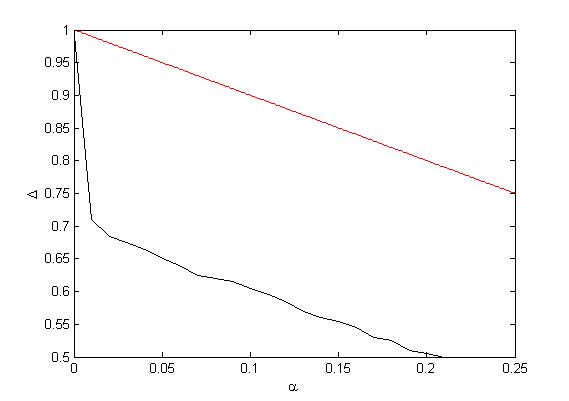}}
\subfigure[$w=0.75$]{
\includegraphics[scale=0.25]{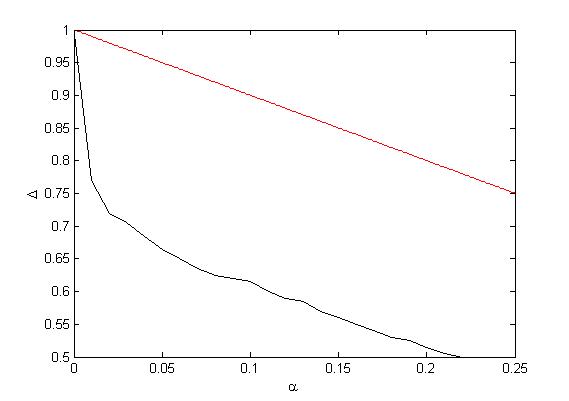}}
\subfigure[$w=0.9$]{
    \includegraphics[scale=0.25]{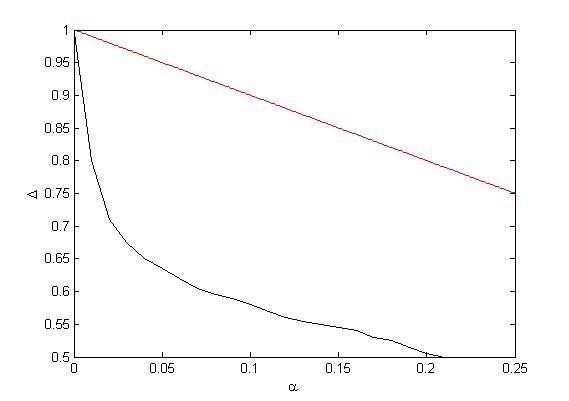}}
  \caption{Bifurcation diagram for extremists and undecided. The curly line separates the
  region of convergence to total consensus (above) from convergence to a partial consensus
  with two components. The straight line is the sufficient condition in \coref{coro-cie}.}
  \label{fig-cie-1}
  \end{center}
\end{figure}

We also plot the center of masses of the first
half of the distribution to show that it is not a
smooth function of $\alpha$ and that close to the
critical value $\Delta_c(\alpha)$ there is a
jump. We did this for the previous 3 values of
$w$ but show only one result for brevity.

\begin{figure}[h!]
  \centering
    \includegraphics[scale=0.35]{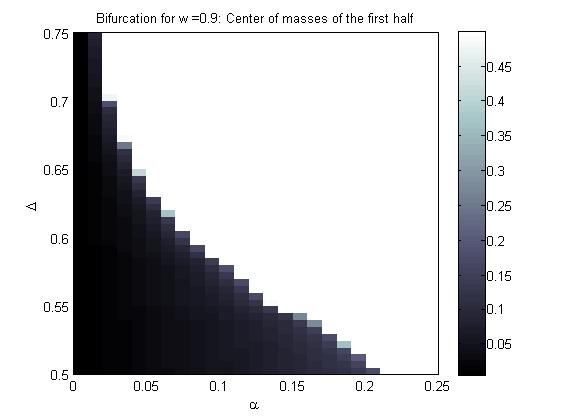}
  \caption{$w = 0.9$. Center of masses of the first half, showing that the transition is abrupt.}
  \label{ZoomFirstHalf-09}
\end{figure}

\subsubsection{Initial Uniform Conditions, Impact of $\Delta$}
We present here the evolution of the number of
components with respect to $\Delta$, using as
initial condition a uniform distribution. Note
that we have capped the situations with more than
7 components into the category ''7 or more'',
which are represented by 7 in the graph. For a
component to be considered as such, we require
that it has at least 1\% of the total mass.
Otherwise we consider it as a zero. Again, the
results are plotted for the 3 different values of
$w$.

We observe that the results are almost independent of $w$, as
there is almost no difference between the 3 curves (see
\fref{NumberOfcomponentsTogether} for the combined plot of all
3 functions). Another interesting thing to remark is that if we
compare our results for $w = 0.5$ with the deterministic model
with the ones in Ref~\refcite{deffuant2000mixing} with the
probabilistic model, the intervals of $\Delta$ in which they
have a high probability of convergence to $n$ components
correspond to the same intervals in which we have convergence
to $n$ components. This suggests that the approximation for $N
= \infty$ is good enough to preserve properties such as the
final state.

%[(Javi): Careful with the color here. The paper is probably going to be B/W]

\begin{figure}[ht]
  \centering
    \includegraphics[scale=0.35]{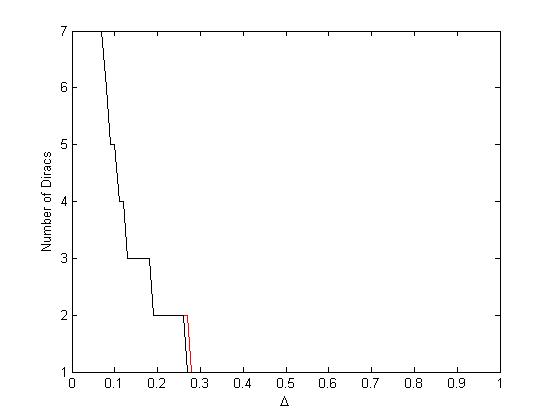}
  \caption{$\Delta$ vs Number of components of $m(\infty)$. Uniform initial conditions. Blue - $w = 0.5$ (below black), Red - $w = 0.75$, Black - $w = 0.9$}
  \label{NumberOfcomponentsTogether}
\end{figure}

\subsubsection{Beta Distribution as Initial Condition}

Here we study the evolution of the number of
components with respect to $\Delta$, using as
initial condition a Beta(1,6) distribution. The
functions that have 5 or more components have
been put into the category represented with a 5.
Again, we consider a component if it has 1\% of
the total mass or more. We present the results
for the 3 different values of $w$.

\begin{figure}
\begin{center}
\subfigure[$w=0.5$]{
\includegraphics[scale=0.25]{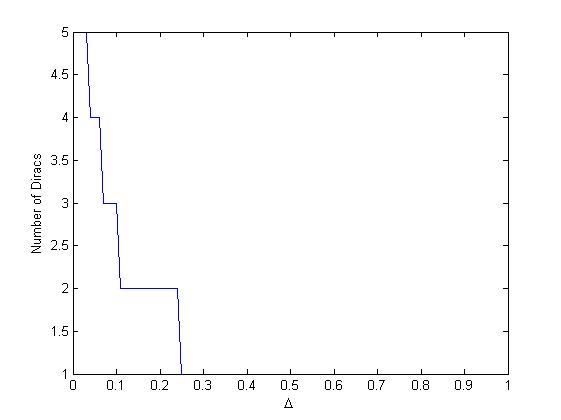}}
\subfigure[$w=0.75$]{
    \includegraphics[scale=0.25]{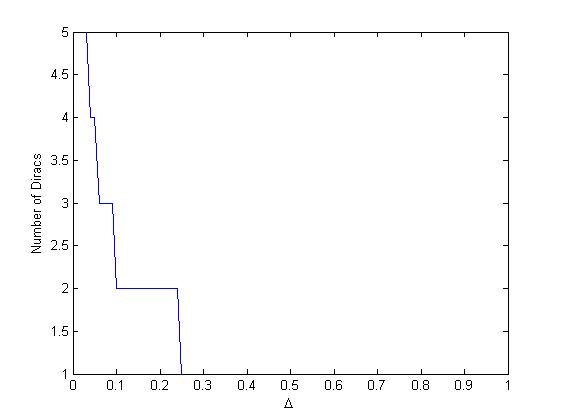}}
\subfigure[$w=0.9$]{
    \includegraphics[scale=0.25]{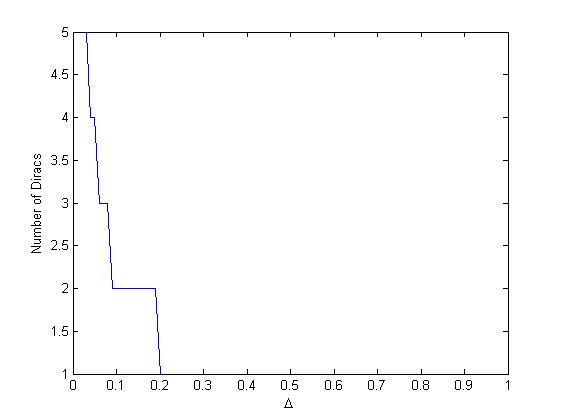}}
  \caption{$\Delta$ vs Number of components of $m(\infty)$. Initial condition Beta(1,6).}
  \label{NumberOfDiracsBeta-09}
\end{center}
\end{figure}

We can observe again the same phenomenon as in the uniform
case, namely that the influence of $w$ is negligible. If we
compare the results from the ones in Subsection 7.3, we can
conclude that the final result depends on the initial
condition, even for the same parameters $w$ and $\Delta$.
Moreover, we can see that for a fixed $(w,\Delta)$, if we start
with a Beta distribution, the number of components will be
smaller or equal than if we start with a uniform one. This is
explained by the fact that with the Beta distribution the mass
is more concentrated than with the Uniform distribution (in our
case: to the left) and therefore it should be harder (i.e,
$\Delta$ should be smaller) to split in the same number of
components.

\appendix
% manual edition of section header in appendix
% to avoid akward theorem numbering, the section label in appendix does
% not contain the word "Appendix". I have to set it manually here
\setcounter{section}{1}
\section*{Appendix \thesection . Probabilistic, Topological and Measurability issues}
\label{sec-probab}

In the particular case of probability measures on $[0,1]$,
a sequence $\nu_n$ converges weakly to $\nu$ if
and only if $\cro{f,\nu_n}$ converges to $\cro{f,\nu}$ for
any continuous (and hence bounded) $f: [0,1] \to \Reals$.
Equivalently, the  cumulative distribution function (CDF) of
$\nu_n$ converges to the CDF of $\nu$ at all
continuity points of the limit.

More generally, Ethier-Kurtz\cite{ethier1986markov} will be the main reference.

Let $\calS$ be a metric space with a $\sigma$-field (not necessarily the Borel $\sigma$-field),
$\calP(\calS)$ the space of probability measures on $\calS$ (for this $\sigma$-field),
and
$D(\Reals_+,\mathcal{S})$ the Skorohod space of right-continuous paths with left-hand limits
(for this metric).

When $\calS$ is given the Borel $\sigma$-field,
the weak topology of $\calP(\calS)$ corresponds to the convergences
\[
P_n \xrightarrow[n\to\infty]{\text{weak}} P
\Leftrightarrow
\langle f,P_n \rangle \xrightarrow[n\to\infty]{} \langle f,P \rangle\,,\;
\forall f\in C_b(\mathcal{S},\Reals)
\]
where $C_b(\mathcal{S},\Reals)$ denotes the space of continuous bounded functions.
Convergence in law of random elements,
defined possibly on distinct probability spaces but having
common sample space $\calS$, is defined as weak convergence of their laws:
\[
Y_n \xrightarrow[n\to\infty]{\text{law}} Y
\Leftrightarrow
\calL(Y_n) \xrightarrow[n\to\infty]{\text{weak}} \calL(Y)
\Leftrightarrow
\E(f(Y_n)) \xrightarrow[n\to\infty]{} \E(f(Y))\,,\;
\forall f\in C_b(\mathcal{S},\Reals)\,.
\]

If $\calS$ is separable and is given the Borel $\sigma$-field,
then the weak topology is metrizable  and $\calP(\calS)$ is separable
(Ref.~\refcite{ethier1986markov}, Theorems 3.3.1 and 3.1.7).

If $\calS$ is not separable, then the Borel $\sigma$-field is usually too strong
to sustain reasonable probability measures, and $\calS$ must be given a weaker, separable,
$\sigma$-field. This causes problems between topological and measure-theoretic issues, and
classic results such as the Portmanteau theorem (Ref.~\refcite{ethier1986markov}, Theorem 3.3.1)
may fail to hold.

The natural $\sigma$-field on $D(\Reals_+,\mathcal{S})$ is the
product (or projection) $\sigma$-field of the $\sigma$-field on $\mathcal{S}$,
and will always be used in the sequel.
The classical topology given $D(\Reals_+,\mathcal{S})$ is the Skorohod topology,
which can be metrized by (3.5.2) or (3.5.21) in Ref.~\refcite{ethier1986markov}.
If $\calS$ is separable then $D(\Reals_+,\mathcal{S})$
is separable (Ref.~\refcite{ethier1986markov}, Theorem 3.5.6) and
then, if $\calS$ is given the Borel $\sigma$-field,
the Borel $\sigma$-field of the Skorohod topology and the product $\sigma$-field
coincide.
For weak convergence with a continuous limit process,
uniform convergence on bounded time intervals
may be used with adequate measurability assumptions on the test functions
(Ref.~\refcite{ethier1986markov}, Theorem 3.10.2).

% manual edition of section header in appendix
% to avoid akward theorem numbering, the section label in appendix does
% not contain the word "Appendix". I have to set it manually here
\addtocounter{section}{1}
\setcounter{subsection}{0}
\section*{Appendix \thesection . Algorithm}
\label{sec-algo}
In this section we present an algorithm for the numerical solution of
Equation \eqref{eq-pdf}. The algorithm takes as input
the initial condition $f^r(x,0)$, assumed to be a
piecewise constant function and the time horizon $T$
%time $T$
 up to which which we want to calculate an
approximate solution. %and a maximum error
%$\varepsilon$ in total variation norm.
 It outputs an approximation of the
solution $f^r(x,T)$. It works as follows.

First, in
steps of $\Delta t$ we approximate $f^r(x,t +
\Delta t)$ by using a forward Euler method:
\ben
f^e(x,t + \Delta t) \eqdef f^r(x,t) + \Delta t \partial_t f^r(x,t)
\een
Here we exploit the fact that $f^r(x,t)$ is a
piecewise constant function, so that we can
calculate analytically the derivative which is a
piecewise linear function. The expression for the derivative is explained later.
Hence, $f^e(x,t + \Delta t)$ is also piecewise
linear (in $x$), as it is the sum of a piecewise linear
and a piecewise constant function. Then, we
approximate $f^e(x,t + \Delta t)$ with another
piecewise constant function (which we will call
$f^r(x,t + \Delta t)$ for simplicity) of
%$I_{t+\Delta t}$
%\color{red}
$I$
%\color{black}
 intervals, so that we can reuse
the same scheme and we can compute explicitly the
expression for the derivative.
%The constants are chosen in order to minimize the $L^{1}$ norm of the error
%(or, equivalently, the total variation norm of its associated measure).
%\color{red}
The constants are chosen so that the integral of $f^r(.,t)$ is equal to 1 (i.e. it is probability density).
%\color{black}
%[(Javi): $L^{1}$ error is equivalent to total variation norm. I don't know which one is easier for the reader to understand, as TV needs to have a %measure associated to $f(x,t)$ for each $f$. Probably the best is to keep both and state/prove they are equivalent]
We perform this loop until we calculate
$f^r(x,T)$ in steps of $\Delta t$. %The number of iterations and the number of intervals in the piecewise constant approximations are computed such that the total error is less than the specified error bound.
The algorithm is given next.
%
%
%Knowing beforehand the complexity, we can choose
%the parameters $\Delta t$ and $I_t$ so that the
%total error is less than the specified bound.
%
%The first way consists on having a constant
%number of intervals throughout the algorithm.
%Although the internal loop is executed faster
%(only once), we might overestimate the number of
%intervals at some time, where the equation is not
%stiff enough or $\Delta t$ is very small.
%
%In contrast in a second way, if we decide to adapt the number of
%intervals at each step so that we bound the
%maximum error per iteration, we are sure that we
%won't have more than the necessary intervals, but
%at the cost of possibly having to recalculate
%$f^r(x,t)$ several times, when errors are big.
%
%The asymptotic cost of both algorithms
%is the same, as the bottleneck is not the
%calculation of $f^r(x,t)$ but that of $f^e(x,t)$.
%Both algorithms are given next.
%
%\color{red}
%\LARGE
%I've changed Input and State from the algorithm (see tex), but I'm not sure whether to write them as before or not.
%\normalsize
\begin{algorithm}[H]
\caption{Numerical Solution of Equation \eqref{eq-pdf}}
\begin{algorithmic}
 %\State \textbf{Input}  $f^r(x,0),T,\varepsilon_{max}$
 \State \textbf{Input}  $f^r(x,0),T,\Delta t, I$
 \State \textbf{Output} $f^r(x,T)$
 %\State Pick $\Delta t$ and $I$ according to $\varepsilon_{max}$\;

 \For{$t \leftarrow 0$ \textbf{to} $T$ \textbf{step} $\Delta t$}
    \State  $f^e(x,t+\Delta t) \leftarrow f^r(x,t) + \Delta t \partial_t f^r(x,t)$
    \State  $f^r(x,t+\Delta t) \leftarrow $ \pro{PiecewiseConstantApproximation}($f^e(x,t+\Delta t),I$)
 \EndFor
\end{algorithmic}
\end{algorithm}

%\color{black}

%\begin{algorithm}[H]
%\caption{Numerical Solution of Equation \eqref{eq-pdf}}
%\begin{algorithmic}
%\State
%\textbf{Input} $f^r(x.0),T,\varepsilon_{max},\Delta
% t$
%\State \textbf{Output} $f^r(x,T)$
%
% \For{$t \leftarrow 0$ \textbf{to} $T$ \textbf{step} $\Delta t$}
%   \State $f^e(x,t+\Delta t) \leftarrow f^r(x,t) +
%     \Delta t \partial_t f^r(x,t)$
%   \State $I \leftarrow 1$
%   \Repeat
%     \State
%        $f^r(x,t+\Delta t) \leftarrow $ \pro{PiecewiseConstantApproximation}($f^e(x,t+\Delta t),I$)
%      \State
%      $\varepsilon_{curr} \leftarrow $ \pro{GetError}($f^r(x,t+\Delta t),f^e(x,t + \Delta
%      t)$)
%      \State
%      $I \leftarrow 2I$
%   \Until{$\varepsilon_{curr} < \varepsilon_{max}$ }
%   \EndFor
%\end{algorithmic}
%\end{algorithm}

The method \pro{PiecewiseConstantApproximation} returns
%the best piecewise constant approximation for a piecewise linear function in terms of minimizing the $L^1$ norm of the functions (or the total variation norm of the associated measures)
%\color{red}
a piecewise constant approximation such that the total integral equals 1.
%\color{black}
% while the method \pro{GetError} returns the error made by such approximation.

%\subsection{Optimal $f^r(x,t)$}
\subsection{Piecewise constant approximation}
%\color{red}
%\LARGE
% Changed title of subsection
%\normalsize
%We want to determine
%which is
%the optimal
We choose as
piecewise constant approximation for $f^{e}(x,t)$ %and have the following proposition:
on any interval $X = [x_s,x_e]$, %$f^{e}\left(\frac{x_s + x_e}{2},t\right)$.
$M = \frac{\int_{X}f^{e}(x,t)dx}{x_e - x_s}$, i.e. the center of mass.
%\color{black}
%\color{red}
%\begin{proposition}
%\label{prop-opt-piece}
%The optimal constant which minimizes the error
%A choice for a constant $M$ which preserves the integral
%on any interval $X = [x_s,x_e]$ is given by %$f^{e}\left(\frac{x_s + x_e}{2},t\right)$.
%$\frac{\int_{X}f^{e}(x,t)dx}{x_e - x_s}$, i.e. the center of mass.
%\end{proposition}
%\color{black}
%The following proposition bounds the error of the method.

\begin{proposition}
\label{prop-num-integral}
For any $t\ge0$ it holds that
$\int_0^1 f^{r}(x,t)\,dx = 1$.
\end{proposition}

\subsection{Analytical expression of $\partial_t f^r(x,t)$}

Now we will give an exact expression for the derivative, using the fact that $f^r(x,t)$ is piecewise constant. This helps to understand how the calculation of the derivative is implemented and its asymptotic cost. We can write, for any $t$, that
$$ f^r(x,t) = \sum_{i=1}^{I}a_i[H(x - x_{i+1}) - H(x - x_i)]$$
where $H(x)$ is the Heaviside step function. Let us set, for any $x_{i}$ and $x_{j}$, that
\begin{align*}
I_1^{i,j}(x) &\eqdef \int_{x-\Delta}^{x+\Delta}H(x-x_{i})H(z-x_{j})dz
= \int_{-\Delta}^{\Delta}H(x-x_{i})H(x+u-x_{j})du\,,
\\
I_2^{i,j}(x) &\eqdef
\frac{1}{w}\int_{x-w\Delta}^{x+w\Delta}H(z-x_{i})H\left(\frac{x-(1-w)z-wx_{j}}{w}\right)dz\\
&=
\int_{-\Delta}^{\Delta}H(x+wu-x_{i})H(x-(1-w)u-x_{j})du\,.
\end{align*}
The expression of $I_1^{i,j}(x)$ and $I_2^{i,j}(x)$ depends on
the relative order between $x_{i}$ and $x_{j}$ and $ m =
\max{\{(1-w)x_{i} + wx_{j},x_{i} - w\Delta\}}$
and is summarized
in Tables \ref{I1} and \ref{I2}.  Finally, we can calculate
$\partial_t f^r(x,t)$ as:
\begin{align*}
 \partial_t f^r(x,t) &=-2\sum_{i,j}a_ia_j(I_1^{i,j}(x) + I_1^{i+1,j+1}(x) - I_1^{i,j+1}(x) - I_1^{i+1,j}(x)) \\
&\qquad + 2\sum_{i,j}a_ia_j(I_2^{i,j}(x) + I_2^{i+1,j+1}(x) - I_2^{i,j+1}(x) - I_2^{i+1,j}(x))\,.
 \end{align*}
%
%
%%% LEB June 9 2010
 %\section{Tables}

\begin{table}[ht]

\footnotesize
\centering
\begin{tabular}{c|c}
  \hline
 \emph{  Case} & $I_1^{i,j}(x)$ \\
  \hline\hline
  $x_{i} \leq x_{j} - \Delta \leq x_{j} + \Delta$ &
  $\left\{
  \begin{array}{cl}
  0 & \text{ if } x \leq x_{j} - \Delta \\
  x - (x_{j} - \Delta) & \text{ if } x_{j} - \Delta \leq x \leq x_{j} + \Delta\\
  2\Delta & \text{ if } x_{j} + \Delta \leq x \\
  \end{array}
  \right.$
   \\
  \hline
  $x_{j} - \Delta \leq x_{i} \leq x_{j} + \Delta$ &
  $\left\{
  \begin{array}{cl}
  0 & \text{ if } x \leq x_{i}  \\
  x - (x_{j} - \Delta) & \text{ if } x_{i} \leq x \leq x_{j} + \Delta\\
  2\Delta & \text{ if } x_{j} + \Delta \leq x \\
  \end{array}
  \right.$
   \\
  \hline
  $x_{j} - \Delta \leq x_{j} + \Delta \leq x_{i}$ &
  $\left\{
  \begin{array}{cl}
  0 & \text{ if } x \leq x_{i} \\
  2\Delta & \text{ if } x_{i} \leq x \\
  \end{array}
  \right.$
   \\
  \hline
\end{tabular}
%\begin{tabular}{cc}
%   \hline
%  % after \\: \hline or \cline{col1-col2} \cline{col3-col4} ...
% \emph{ Value of $I_1^{i,j}(x)$ }& \emph{Condition} \\
%  \hline
%  $0$& $x_i \leq x_j -\Delta$ and $x \leq x_j -\Delta$ \\
%  $0$ & $ x_i \geq x_j -\Delta$ and $x_i \geq x$\\
%  $x-x_j+\Delta$ & $s$\\
%  $2 \Delta$ & b  \\\hline
%\end{tabular}
\caption{Expression for $I_1^{i,j}(x)$}
\label{I1}
\end{table}

\begin{table}[h!]
\footnotesize
\centering
\begin{tabular}{c|c}
  \hline
  \emph{  Case} & $I_2^{i,j}(x)$ \\
  \hline\hline
  $\barr{r} m \leq x_{i} + w\Delta 
  \leq x_{j} - (1-w)\Delta  \leq \\x_{j} + (1-w)\Delta
  \earr $ 
  & % Case 1
  $\left\{
  \begin{array}{cl}
  0 & \text{ if } x \leq x_{j}-(1-w)\Delta \\
  \frac{x - x_{j}}{1-w} + \Delta & \text{ if } x_{j} - (1-w)\Delta \leq x \leq x_{j} + (1-w)\Delta\\
  2\Delta & \text{ if } x_{j} + (1-w)\Delta \leq x \\
  \end{array}
  \right.$
   \\
  \hline
  $\barr{r} x_{i} + w\Delta \leq m \leq x_{j} - (1-w)\Delta \\ \leq x_{j} + (1-w)\Delta\earr $ & % Case 2
  $\left\{
  \begin{array}{cl}
  0 & \text{ if } x \leq x_{j}-(1-w)\Delta \\
  \frac{x - x_{j}}{1-w} + \Delta & \text{ if } x_{j} - (1-w)\Delta \leq x \leq x_{j} + (1-w)\Delta\\
  2\Delta & \text{ if } x_{j} + (1-w)\Delta \leq x \\
  \end{array}
  \right.$
   \\
  \hline
  $\barr{r}m \leq  x_{j} - (1-w)\Delta \leq x_{i} + w\Delta \\  \leq x_{j} + (1-w)\Delta\earr $ & % Case 3
  $\left\{
  \begin{array}{cl}
  0 & \text{ if } x \leq x_{j}-(1-w)\Delta \\
  \frac{x - x_{j}}{1-w} - \frac{x_{i} - x}{w} & \text{ if } x_{j} - (1-w)\Delta \leq x \leq x_{i} + w\Delta\\
  \frac{x - x_{j}}{1-w} + \Delta & \text{ if } x_{i} + w\Delta \leq x \leq x_{j} + (1-w)\Delta\\
  2\Delta & \text{ if } x_{j} + (1-w)\Delta \leq x \\
  \end{array}
  \right.$
   \\
  \hline
  $\barr{r} x_{i} + w\Delta \leq  x_{j} - (1-w)\Delta \leq m  \\ \leq x_{j} + (1-w)\Delta\earr $ & % Case 4
  $\left\{
  \begin{array}{cl}
  0 & \text{ if } x \leq m \\
  \frac{x - x_{j}}{1-w} + \Delta & \text{ if } m \leq x \leq x_{j} + (1-w)\Delta\\
  2\Delta & \text{ if } x_{j} + (1-w)\Delta \leq x \\
  \end{array}
  \right.$
   \\
  \hline
  $\barr{r}m \leq  x_{j} - (1-w)\Delta \\ \leq x_{j} + (1-w)\Delta \leq x_{i} + w\Delta\earr $ & % Case 5
  $\left\{
  \begin{array}{cl}
  0 & \text{ if } x \leq x_{j}-(1-w)\Delta \\
  \frac{x - x_{j}}{1-w} - \frac{x_{i} - x}{w} & \text{ if } x_{j} - (1-w)\Delta \leq x \leq x_{j} + (1-w)\Delta\\
  \Delta - \frac{x_{i} - x}{w} & \text{ if } x_{j} + (1-w)\Delta \leq x \leq x_{i} + w\Delta\\
  2\Delta & \text{ if } x_{i} + w\Delta \leq x \\
  \end{array}
  \right.$
   \\
  \hline
  $\barr{r} x_{i} + w\Delta \leq  x_{j} - (1-w)\Delta \\ \leq x_{j} + (1-w)\Delta \leq m\earr $ & % Case 6
  $\left\{
  \begin{array}{cl}
  0 & \text{ if } x \leq m \\
  2\Delta & \text{ if } m \leq x \\
  \end{array}
  \right.$
   \\
  \hline
  $\barr{r} x_{j} - (1-w)\Delta \leq m \leq x_{i} + w\Delta \\ \leq x_{j} + (1-w)\Delta\earr $ & % Case 7
  $\left\{
  \begin{array}{cl}
  0 & \text{ if } x \leq m \\
  \frac{x - x_{j}}{1-w} - \frac{x_{i} - x}{w} & \text{ if } m \leq x \leq x_{i} + w\Delta\\
  \frac{x - x_{j}}{1-w} + \Delta & \text{ if } x_{i} + w\Delta \leq x \leq x_{j} + (1-w)\Delta\\
  2\Delta & \text{ if } x_{j} + (1-w)\Delta \leq x \\
  \end{array}
  \right.$
   \\
  \hline
  $\barr{r} x_{j} - (1-w)\Delta \leq x_{i} + w\Delta \leq m \\ \leq x_{j} + (1-w)\Delta\earr $ & % Case 8
  $\left\{
  \begin{array}{cl}
  0 & \text{ if } x \leq m \\
  \frac{x - x_{j}}{1-w} + \Delta & \text{ if } m \leq x \leq x_{j} + (1-w)\Delta\\
  2\Delta & \text{ if } x_{j} + (1-w)\Delta \leq x \\
  \end{array}
  \right.$
   \\
  \hline
  $\barr{r} x_{j} - (1-w)\Delta \leq m \\ \leq x_{j} + (1-w)\Delta \leq x_{i} + w\Delta\earr $ & % Case 9
  $\left\{
  \begin{array}{cl}
  0 & \text{ if } x \leq m \\
  \frac{x - x_{j}}{1-w} + \frac{x_{i} - x}{w} & \text{ if } m \leq x \leq x_{j} + (1-w)\Delta\\
  \Delta - \frac{x_{i} - x}{w} & \text{ if } x_{j} + (1-w)\Delta \leq x \leq x_{i} + w\Delta\\
  2\Delta & \text{ if } x_{i} + w\Delta \leq x \\
  \end{array}
  \right.$
   \\
  \hline
  $\barr{r} x_{j} - (1-w)\Delta \\ \leq x_{i} + w\Delta \leq x_{j} + (1-w)\Delta \leq m\earr $ & % Case 10
  $\left\{
  \begin{array}{cl}
  0 & \text{ if } x \leq m \\
  2\Delta & \text{ if } m \leq x \\
  \end{array}
  \right.$
   \\
  \hline
  $\barr{r} x_{j} - (1-w)\Delta \leq  x_{j} + (1-w)\Delta \\ \leq m \leq x_{i} + w\Delta\earr $ & % Case 11
  $\left\{
  \begin{array}{cl}
  0 & \text{ if } x \leq m \\
  \Delta - \frac{x_{i} - x}{w} & \text{ if } m \leq x \leq x_{i} + w\Delta\\
  2\Delta & \text{ if } x_{i} + w\Delta \leq x \\
  \end{array}
  \right.$
   \\
  \hline
  $\barr{r} x_{j} - (1-w)\Delta \\ \leq  x_{j} + (1-w)\Delta \leq x_{i} + w\Delta \leq m\earr $ & % Case 12
  $\left\{
  \begin{array}{cl}
  0 & \text{ if } x \leq m \\
  2\Delta & \text{ if } m \leq x \\
  \end{array}
  \right.$
   \\
  \hline
\end{tabular}
\caption{Expression for $I_2^{i,j}(x)$}
\label{I2}
\end{table}
\normalsize

%
%%% END LEB June 9 2010
%
\subsection{Error Bound}
%
%
%We are interested in estimating the
%error that we are making while approximating $f(x,T)$ by $f^r(x,T)$. Again, we will use as metric the distance %between the associated Lebesgue-Stieltjes measures. Let $\nu_{e}^{t}(x)$ and $\nu_{r}^{t}(x)$ be the Lebesgue-Stieltjes measures associated to %$f^e(x,t)$ and $f^{r}(x,t)$ respectively. We also need to define:
%
To calculate the error made by our approximation, define
$$ g^{s}(x,t) \eqdef f(x,t) \quad \mif t \geq s\geq 0,
\quad g^{s}(x,t) \eqdef  f^r(x,t) \mif 0\leq t<s,$$
and let $\nu_{e}^{t}(dx)$, $\nu_{r}^{t}(dx)$ and
$\mu_{s}^{t}(dx)$ be the measures associated to $f^{e}(x,t)$,
$f^{r}(x,t)$ and $g^{s}(x,t)$ respectively. Note that
$\nu_{r}^{t}(dx) = \mu_{t}^{t}(dx)$. Thus, we want to
bound
\begin{multline*}
 \varepsilon_{tot} = |\mu_{0}^{T}(dx) - \nu_{r}^{T}(dx) |_{T}
= \left|\sum_{k=1}^{T/(\Delta t)}\mu_{(k-1)\Delta t}^{T}(dx) - \mu_{k\Delta t}^{T}(dx)\right|_{T} \\
\leq \sum_{k=1}^{T/(\Delta t)}|\mu_{(k-1)\Delta t}^{T}(dx) - \mu_{k\Delta t}^{T}(dx)|_{T}.
\end{multline*}

We can bound the error done in each iteration of the loop by decomposing it as
\begin{multline*}
 |\mu_{k \Delta t}^{k \Delta t}(dx) - \mu_{(k-1)\Delta t}^{k\Delta t}(dx)|_{T}
 \\
 \leq |\nu_{r}^{k \Delta t}(dx) - \nu_{e}^{k \Delta t}(dx)|_{T}
+  |\nu_{e}^{k \Delta t}(dx) - \mu_{(k-1)\Delta t}^{k\Delta t}(dx)|_{T}
 \eqdef \varepsilon_{c.s} + \varepsilon_{eu}.
\end{multline*}
%
%
%
%Let $I_0$ be the smallest $I$ such that $w\Delta$ and $(1-w)\Delta$ and $\Delta$ are multiples of $\frac{1}{I}$. Assuming that $I$ is large enough to %be a multiple of $I_0$, we can prove the following.
%
%\begin{proposition}
%It holds that
%\begin{equation}
%\label{error_truncation_1}
%\varepsilon_{c.s}(I) = \varepsilon_{c.s}(I_0)\frac{I_0}{I} \leq \frac{\Delta t}{2}|\partial_t f^r(x,(k-1)\Delta t)|_{\infty} \frac{I_0}{I}.
%\end{equation}
%\label{prop-const_spline_proport}
%\end{proposition}

%We will now bound $|\partial_t f^r(x,(k-1)\Delta t)|_{\infty}$.

%\begin{proposition}
%Let $M(t) = |f^r(x,t)|_{\infty}$. If $|f^r(x,0)|_{\infty} = M(0) = M < \infty$
%then the following uniform bound holds:
%$$ |\partial_t f^r(x,k\Delta t)|_{\infty} \leq C_1(M,T)\,, \quad \forall \; 0 \leq k \leq \frac{T}{\Delta t} - 1\,.$$
%\label{prop-const_spline_deriv_bound}
%\end{proposition}
%
%Substituting in \eqref{error_truncation_1}, we obtain that
%
%\begin{equation}
%\label{error_truncation_2}
%\varepsilon_{c.s} \leq \frac{C_1I_0}{2}\frac{\Delta t}{I} = O\left(\frac{\Delta t}{I}\right).
%\end{equation}
%
%\color{red}
%\LARGE
%Removed text.

%%\normalsize
%Let us assume we know \textit{a posteriori} that $\varepsilon_{c.s}$ has the form $ c \Delta t$, where $c$ is a small constant (in practice, numerical experiments have shown that $c$ is of the order $10^{-6}-10^{-7}$). Then, with the known fact that
%\color{black}
%\color{blue}
%%\LARGE
%Alternative formulation. Substitute the previous red sentence with the blue part (including proofs), whatever suits best.
%
%\normalsize

\begin{proposition}
\label{prop-err_ecs}
Let $M(t) = |f^r(x,t)|_{\infty}$. If $|f^r(x,0)|_{\infty} = M(0) = M < \infty$
then the following uniform bound holds:
$$ |\partial_t f^r(x,k\Delta t)|_{\infty} \leq c(M,T)\,, \quad \forall \; 0 \leq k \leq \frac{T}{\Delta t} - 1\,.$$
This results in the following bound for $\varepsilon_{c.s}$:
\begin{equation}
\label{error_truncation_2}
\varepsilon_{c.s} \leq c(M,T)\Delta t.
\end{equation}
\end{proposition}
The constant $c(M,T)$ needs to be evaluated empirically; in practice, numerical experiments have shown that it is of the order $10^{-6}-10^{-7}$. Then, with the known fact that
%\color{black}
\begin{proposition}
\begin{equation}
\label{error_euler}
\varepsilon_{eu} = O((\Delta t)^2),
\end{equation}
\label{prop-error-euler}
\end{proposition}
%
%\color{red}
using Equations (\ref{error_truncation_2}) and (\ref{error_euler}) yields
%\color{black}
%
$$ |\mu_{k \Delta t}^{k \Delta t}(dx) - \mu_{(k-1)\Delta t}^{k\Delta t}(dx)|_{T} \leq \varepsilon_{c.s} + \varepsilon_{eu}
= c\Delta t + O\left((\Delta t)^2\right).$$
Finally, we  bound $|\mu_{(k-1)\Delta t}^{T}(dx) - \mu_{k\Delta t}^{T}(dx)|_{T}$ in terms of $|\mu_{k \Delta t}^{k \Delta t}(dx) - \mu_{(k-1)\Delta t}^{k\Delta t}(dx)|_{T}$.
\begin{proposition}
For all $1 \leq k \leq \frac{T}{\Delta t}$ and for all $t \geq
k\Delta t$ we have that
\bearn \left|\mu_{k\Delta t}^{t}(dx) - \mu_{(k-1)\Delta
t}^{t}(dx)\right|_{T}
  &\leq &e^{8(t-k\Delta t)}\left|\mu_{k \Delta t}^{k \Delta t}(dx) -
\mu_{(k-1)\Delta t}^{k\Delta t}(dx)\right|_{T}.
 \eearn
%
%\bearn \lefteqn{\left|\mu_{k\Delta t}^{t}(dx) -
%\mu_{(k-1)\Delta t}^{t}(dx)\right|_{T} =
%\int_{0}^{1}|g^{k\Delta t}(x,t) -
%g^{(k-1)\Delta t}(x,t)|dx }\\
%  &\leq &
%  e^{8(t-(k-1)\Delta t)}\int_{0}^{1}|g^{k\Delta
%t}(x,(k-1)\Delta t) - g^{(k-1)\Delta t}(x,(k-1)\Delta t)|dx
% \\
% & =& e^{8(t-(k-1)\Delta t)}|\mu_{k \Delta t}^{k \Delta t}(dx) -
%\mu_{(k-1)\Delta t}^{k\Delta t}(dx)|_{T}
% \eearn
\label{prop-error-continuation}
\end{proposition}

Combining the previous propositions, we obtain:
%\color{red}
\begin{theorem}
For any fixed $T$, the error of the method is $C + O(\Delta t)$, where $C$ is a constant that depends on $c$ and $T$.
\label{theo-order-method}
\end{theorem}
%\color{black}

\subsection{Complexity}
%[JYLB: not clear to me why we need the following: ]

%We will now give the complexity analysis of both
%algorithms. For simplicity of the analysis, we
%will assume that $I$ is large enough so that
%$w\Delta$ and $(1-w)\Delta$ and $\Delta$ are
%multiples of $\frac{1}{I}$.
%\color{red}
We will now give the complexity analysis of the algorithm. The computation of the derivative takes $O(I^2)$, where $I$ is the number of intervals, since there is a double sum over $I$ intervals. Also, this produces $O(I^2)$ splines because every $I_{k}^{i,j}(x), k=1,2$ is composed of at most 4 splines. Since the splines are not produced in increasing order of $x$, we need to sort them, which takes $O(I^2 \log I)$ time. %Taking into account the expression of the derivative, and the assumption on $I$, the support of every spline is the union of some of the intervals, i.e, there isn't any spline such that its support doesn't fully cover some interval. Therefore, we can compress our $O(I^2)$ splines into $O(I)$ splines in one pass $(O(I^2)$ time).
Finally, we only need one pass to make the piecewise constant spline approximation since now everything is sorted. This takes $O(I^2)$ time. Since all this loop is executed $ \frac{T}{\Delta t}$ times,
the running time has complexity $ O\left(\frac{1}{\Delta t}I^2
\log I\right)$.
\addtocounter{section}{1}
\setcounter{subsection}{0}
\setcounter{theorem}{0}
\section*{Appendix \thesection . Proofs}
\label{sec-proof}

%\subsection{\sref{sec-finite} proofs}

\subsection{Proof of \pref{lem-convex}}

By definition, since $h$ is convex,
\begin{align*}
h(wx+(1-w)y)&\le wh(x) +(1-w)h(y)\,,
\\
h(wy+(1-w)x)&\le wh(y) +(1-w)h(x)\,,
\end{align*}
with strict inequalities if $h$ is strictly convex except when $x= y$ or $w\in\{0,1\}$,
and summing these two inequalities yields the result.

\subsection{Proof of \pref{pro-partcons}}

The first statement is obvious, since a partial consensus is an
absorbing state.

We prove the second statement. It follows from the second
statement in \coref{cor-fin-mom} that, if the two peers, say $(i,j)$ chosen
at any time slot $k'$ are such that $\abs{X^N_i(k') -
X^N_j(k')} \leq \Delta$ and $X^N_i(k') \neq X^N_j(k')$, then
$\mu^N_n(k'+ 1) < \mu^N_n(k')$. Assume now that the hypothesis
of the second statement holds. It follows that all peers chosen
for interaction at times $k'\geq k$ have reputation values that
either differ by more than $\Delta$, or are equal, thus, at any
time slot $k'\geq k$, the interaction has no effect. It follows
that $M^N(k)=M^N(k')$ for $k'\geq k$.

Further, assume that $M^N(k)$ is not a partial consensus. Thus,
there exists a pair of peers $(i,j)$ such that $\abs{X^N_i(k) -
X^N_j(k)} \leq \Delta$ and $X^N_i(k) \neq X^N_j(k)$. The pair
$(i,j)$ is never chosen in a interaction at times $k' \geq k$,
for otherwise this would contradict the fact that $M^N(k')$ is
stationary. But this occurs with probability $0$.

\subsection{Proof of \pref{pro-clusters}}

Let $i$ and $j$ be the peers selected for interaction at time $k+1$. If at time $k$
they were in different clusters, then nothing happens and the proposition holds. Assume now that at time $k$ they were in
the same cluster, say $C_{\ell}$. Let $i'$ be a peer not in $C_{\ell}$ at time $k$;
at time $k+1$ after interaction, the opinions of $i$ and $j$  have moved closer,
hence farther from $i'$ to which they are still not connected.
Hence, the only difference between connections at
time $k$ and $k+1$ concern pairs of peers that that are both in
the same cluster, and the result easily follows.

\subsection{Proof of \thref{thm-longtimelim}}

Let $\sigma^2(k)$ be the variance of
 $M^N(k)$ (we drop superscript $N$ in the
 notation local to this proof).
 %We first show, by
% contradiction,
% that $\sigma(k) \to 0$ with probability 1.
 By \coref{discrete_variance_drop}, $\sigma(k)$ is non-decreasing and non-negative,
 and thus converges to some $\sigma(\infty)$.

 %\paragraph{Step 1.} Assume first that
% $\sigma(\infty)>0$. %Thus there exists some $\alpha >0$
% and a random time $K_1$ such that for all $k>K_1$ : $\sigma > \alpha$.
%
For $k\geq K^N$ the set of clusters remains the same,
$\calC^N(k)=\{C_1, ..,C_{\ell}\}$, and we can thus define the
diameter of cluster $\ell_1 \in \{1,\dots, L^N\}$ by
 \be \delta_{\ell_1}(k)= \max_{i,j\in C_{\ell_1}} \abs{X^N_i(k)-X^N_j(k)}
 \label{eq-proof-max}
 \ee
 and set
 \ben
 \delta_{\ell_1} = \limsup_{k \geq K^N}\delta_{\ell_1}(k)
 \een
Assume that $\delta_{\ell_1} >0$ for some $\ell_1$. Since the sequence
$\sigma^2(k)$ converges, there exists some random
time $K_1\geq K^N$ such that for all $k,k'>K_1$ we have
 \be
 \abs{\sigma^2(k') -
\sigma^2 ( k ) }< \frac{2w(1-w)}{N}
\lp\frac{\delta_{\ell_1}}{2}\rp^2\,,
 \label{eq-pr-cs}
\ee
and there is an
infinite subsequence of time slots $K_2(n) \geq K_1$ for $n \in \Nats$ such that
\bearn
 \delta_{\ell_1}(K_2(n)) &>&\frac{\delta_{\ell_1}}{2}>0\,.
 \eearn

For $k \geq K^N$, let $(I(k), J(k))$ be a pair of peers that
achieves the maximum in \eqref{eq-proof-max} and let $E_k$ be
the event  ``the pair of peers selected for interaction at time
$k$ is $(I(k),J(k))$". The probability of $E_k$, conditional to
all past up to time slot $k$, is $\frac{2}{N(N-1)}$, thus is
constant and positive. Thus the probability that $E_k$ occurs
infinitely often is $1$, \emph{i.e.}, with probability $1$ we can
extract an infinite subsequence of time slots $K_3(n)$ of
$K_2(n)$ such that $E_{K_3(n)}$ is true.
The following lemma then implies that
\bearn
 \sigma^2 \lp K_3(n) +1\rp- \sigma^2 \lp K_3(n) \rp &>&
 \frac{2w(1-w)}{N}\lp
 \frac{\delta_{\ell_1}}{2} \rp^2
 \eearn
which contradicts \eqref{eq-pr-cs}, which proves by contradiction that
$\delta_{\ell_1}= 0 $.

\begin{lemma}
  Let $(i,j)$  be the pair of peers chosen for interaction at time slot
 $k$. Assume that $\abs{X^N_i(k)-X^N_j(k)} \leq
 \Delta$. Then the reduction in variance is
 $\sigma^2(k+1)-\sigma^2(k)=\frac{2w(1-w)}{N}\lp
 X^N_i(k)-X^N_j(k) \rp^2$.
 \end{lemma}

 \begin{proof}
 By direct computation.
 \end{proof}

Let $\mu_{\ell_1}(k)$ be the empirical mean of cluster $\ell_1$
at time $k\geq K^N$. Since interactions that modify the state
of the process at times $k\geq K^N$ are all intra-cluster, it
follows that $\mu_{\ell_1}(k)=\mu_{\ell_1}(K^N):=
\mu_{\ell_1}(\infty)$ for all $k \geq K^N$. For $i\in
C_{\ell_1}$ it holds that $\abs{X^N_i(k)-\mu_{\ell_1}(k)}\leq
\delta_{\ell_1}(k)\to 0$, and hence $X^N_i(k)\to
\mu_{\ell_1}(\infty)$ as $k \to \infty$. Thus, for any
continuous $f: [0,1] \to \Reals$:
 \ben
 \limit{k}{\infty} \cro{f,M^N(k)} = \frac{1}{N}
 \sum_{\ell_1=1}^{L^N} N_{\ell_1}
 f\lp \mu_{\ell_1}(\infty)\rp
 \een where $N_{\ell_1}$ is the cardinality of $C_{\ell_1}$.
This shows that, with probability 1, $M^N(k)$ converges to
$M^N(\infty)=\frac{1}{N}
 \sum_{\ell_1=1}^{L^N} N_{\ell_1}
 \delta_{ \mu_{\ell_1}(\infty)}$.

 It remains to show that $M^N(\infty)$ is a partial consensus. This
 follows from the fact that if $i$ and $j$ are not in the same
 cluster at time slot $k$, then
$
  \abs{X^N_i(k)-X^N_j(k)} > \Delta
$, which implies that
$\abs{\mu_{\ell_1}(k)-\mu_{\ell_2}(k)}>\Delta$ if $\ell_1 \neq
 \ell_2$ and, since, $\mu_{\ell_1}(k)$ is stationary for $k$
 large enough, that
 $\abs{\mu_{\ell_1}(\infty)-\mu_{\ell_2}(\infty)}>\Delta$.

%
%\subsection{Proof of \sref{sec-convergence} proofs}

\subsection{Proof of \thref{theo-eunl}}

We write \eqref{eq-genfam} and \eqref{eq-fm} in the notation of
Section 2.2 in Graham\cite{graham2000chaoticity},
in which the corresponding equations are (2.5) and (2.7), and
\begin{align*}
\calA(\mu)h(x)
&=
2 \big\langle [h(w x + (1-w) y)-h(x)] \ind{|x-y|\le\Delta} \,, \mu(dy) \big\rangle
\\
&= \int(h(z)-h(x)) J(\mu,x,dz)
\end{align*}
for $J(\mu,x,dz)$ the image measure of $\ind{|x-y|\le\Delta} 2\mu(dy)$ by
$y\mapsto w x + (1-w) y$.
Since
$|J(\mu,x,\cdot)| \le 2$ and $|J(\mu,x,\cdot) - J(\nu,x,\cdot)| \le 2|\mu-\nu|$,
the assumptions of
Proposition~2.3 in Ref.~\refcite{graham2000chaoticity} are satisfied, yielding the results.
The family \eqref{eq-genfam} is uniformly bounded by $4$ in operator norm,
and thus there is a well-defined inhomogeneous Markov process with generator $\calA (m(t))$ at time $t$
and arbitrary initial law.

\subsection{Proof of \thref{theo-chaotic-aux}}

First, the proof of \emph{(\ref{chaotic-aux1})}.
The generator $\calA^N$ corresponds to the ``binary mean-field model'' (2.6) in
Graham-M\'el\'eard\cite{graham1997stochastic} with $N$ instead of $n$
and $\calL_i=0$,
and (using $\sum_{1\le i\neq j\le N} = 2 \sum_{1\le i<j\le N}$) ``jump kernel''
\[
\widehat{\mu}(x,y,dh,dk)
= \ind{|x-y|\le\Delta} 2 \delta_{\{(w-1)x + (1-w)y, (w-1)y + (1-w)x\}}(dh,dk)
\]
which is uniformly bounded in total mass by $\Lambda=2$.
We conclude with Theorem~3.1 in Ref.~\refcite{graham1997stochastic}  and
the triangular inequality
$|\frac1N \sum_{i=1}^N \calL(\widehat{X}^N_i) - Q|_T
\le | \calL(\widehat{X}^N_i) - Q|_T$
(the $\widehat{X}^N_i$ are exchangeable).

Now, the proof of \emph{(\ref{chaotic-aux2})}. As in
the proof of Theorem~3.1 in Ref.~\refcite{graham1997stochastic},
\[
\biggl\langle\phi ,
\widehat{\Lambda}^N
- \frac1N \sum_{i=1}^N \calL(\widehat{X}^N_i)
\biggr\rangle^2
=
\frac1{N^2}
\Biggl[\sum_{i=1}^N
(\phi(\widehat{X}_i^N) - \E[\phi(\widehat{X}_i^N)])
\Biggr]^2
\]
in which
\begin{multline*}
\Biggl[\sum_{i=1}^N
(\phi(\widehat{X}_i^N) - \E[\phi(\widehat{X}_i^N)])
\Biggr]^2
=
\sum_{i=1}^N
(\phi(\widehat{X}_i^N) - \E[\phi(\widehat{X}_i^N)])^2
\\
+ \sum_{1\le i \neq j \le N}
(\phi(\widehat{X}_i^N) - \E[\phi(\widehat{X}_i^N)])
(\phi(\widehat{X}_j^N) - \E[\phi(\widehat{X}_j^N)])
\end{multline*}
where the first sum on the r.h.s. has $N$ terms, the second $N(N-1)$, and
\begin{multline*}
\E\bigl[
(\phi(\widehat{X}_i^N) - \E[\phi(\widehat{X}_i^N)])
(\phi(\widehat{X}_j^N) - \E[\phi(\widehat{X}_j^N)])
\bigr]
\\
= \E[\phi(\widehat{X}_i^N)\phi(\widehat{X}_j^N)] - \E[\phi(\widehat{X}_i^N)]\E[\phi(\widehat{X}_j^N)]\,,
\end{multline*}
and we conclude to the first formula in \emph{(\ref{chaotic-aux2})} using \emph{(\ref{chaotic-aux1})} for $k=2$.

Classically,
the weak topology in the Polish space $\calP(D(\Reals_+,[0,1]))$
has a convergence-determining sequence $(g_m)_{m\ge1}$
of continuous functions bounded by~$1$ (such a sequence is constructed in
the proof of Proposition 3.4.4 in Ethier-Kurtz\cite{ethier1986markov}), and can
thus be metrized by
$
d(P,Q) = \bigl(\sum_{i\ge1} 2^{-i}\cro{g_m, P-Q}^2\bigl)^{1/2}
$.
Moreover, the first formula in \emph{(\ref{chaotic-aux2})} and
the second in \emph{(\ref{chaotic-aux1})}
imply that $\E(d(\widehat{\Lambda}^N, Q)^2)$ goes to $0$, which
proves convergence in probability for $\widehat{\Lambda}^N$.

The result for $\widehat{\Lambda}^N$ implies the result for its
marginal process $\widehat{M}^N$ as a quite general topological
fact, since the limit marginal process $m$ is continuous and
the spaces are Polish (Theorem 4.6 in Graham-M\'el\'eard,\cite{graham1997stochastic}
Section 4.3 in M\'el\'eard\cite{meleard1996asymptotic}); proofs first use the
Skorohod topology, and then Theorem 3.10.2 in
Ref.~\refcite{ethier1986markov}.

\subsection{Proof of \thref{theo-aux-resc}}

Let $\lambda_N: \Reals_+\to\Reals_+$ be the (random) time-change given
by the linear interpolation of $\lambda_N(\frac{k}N) = \frac{T_k}N$, \emph{i.e.}, by
\[
t\in\biggl[\frac{k}N , \frac{k+1}N\biggr]
\mapsto
\lambda_N(t) = (k+1-tN)\frac{T_k}N + (tN-k)\frac{T_{k+1}}N\,,
\qquad
k\in\Nats\,.
\]
Then \eqref{eq-rescaux} implies that
\[
\widetilde{X}^N(t) =
\widehat{X}^N(\lambda_N(t))\,,
\qquad t\in\Reals_+\,,
\]
so that their atomic distance is null.
The triangular inequality yields, for $k\in\Nats$,
\[
|\lambda_N(t)-t|\le \biggl|\frac{T_k}N-\frac{k}N\biggr|
+\frac1{N}(T_{k+1}-T_k) + \frac1{N}\,,
\qquad
t\in\biggl[\frac{k}N , \frac{k+1}N\biggr]\,,
\]
and hence, for any $T>0$,
\[
\sup_{0\le t \le T}|\lambda_N(t)-t|
\le \frac1N\sup_{0\le k \le \lfloor NT \rfloor }|T_k-k|
+ \frac1{N}\sup_{0\le k \le \lfloor NT \rfloor }(T_{k+1}-T_k) + \frac1{N}\,.
\]
For $\eps>0$,
Kolmogorov's maximal inequality implies that
\[
\P\biggl(\frac1N\sup_{0\le k \le \lfloor NT \rfloor }|T_k-k| \ge \eps\biggr)
\le
\frac{1}{\eps^2 N^2}
\sum_{i=1}^{\lfloor NT \rfloor}\text{var}(T_{i}-T_{i-1})
=\frac{\lfloor NT \rfloor}{\eps^2 N^2}\,,
\]
and classically
\[
\P\biggl(\frac1N\sup_{0\le k \le \lfloor NT \rfloor}(T_{k+1}-T_k)\ge \eps\biggr)
=
1-(1-\mathrm{e}^{-N\eps})^{\lfloor NT \rfloor+1}
\le (\lfloor NT \rfloor+1)\mathrm{e}^{-N\eps}\,.
\]
Hence, for all $\delta>0$,
\[
\lim_{N\to \infty }\P\biggl(\sup_{0\le t \le T}|\lambda_N(t)-t| \ge \delta\biggr)
=0\,,
\]
from which the result  follows.

\subsection{Proof of \thref{theo-chaotic-resc}}

Result~\emph{(\ref{chaotic-resc1})} follows from the previous convergence in probability result
and \thref{theo-chaotic-aux},
using either the uniform continuity of the test functions (for the atomic metric)
or Corollary~3.3.3 in Ethier-Kurtz\cite{ethier1986markov} (for the usual metric).
Result~\emph{(\ref{chaotic-resc2})}, which involves Polish spaces, follows
as for \thref{theo-chaotic-aux}.

%
%
%\subsection{Proof of \sref{sec-infinite} proofs}

\subsection{Proof of \pref{pro-stdevbd}}
For $0 < b$ and $t\in[0,b]$ define $u(t):=\sigma^2(b-t)-\sigma^2(b)$. Note that $\mu_1(t)$ is a constant thus $u(t)=\mu_2(b-t)-\mu_2(b)$.
By the alternative definition of Problem 1
 \bearn
 u(t)&=&-\int_{b-t}^b\int_{[0,1]^2}
  \lb
    (wx+ (1-w)y)^2+ (wy+ (1-w)x)^2 - x^2 -y^2
  \rb
  \\
  &&
  \ind{\abs{x-y}\leq \Delta}
  m(s)(dx)m(s)(dy)
 ds
 \eearn
 By \pref{lem-convex}, the bracket is nonpositive, and the indicator function is upper bounded by $1$ thus
 \bearn
 u(t) &\leq &-\int_{b-t}^b\int_{[0,1]^2}
  \lb
    (wx+ (1-w)y)^2+ (wy+ (1-w)x)^2 - x^2 -y^2
  \rb\\
  &&
  m(s)(dx)m(s)(dy)
 ds\\
 & = & K \int_{b-t}^t \sigma^2(s) ds
 = K\lp \sigma^2(b) + \int_{0}^t u(s) ds\rp
 \eearn
 with $K=4 w(1-w)$. By Gr\"onwall's lemma:
 \bearn
 u(t) &\leq& K \sigma^2(b) t + K^2 \sigma^2(b) e^{Kt}\int_0^ts e^{-Ks}ds
   =  \sigma^2(b) \lp e^{Kt}-1\rp
 \eearn
Let $t=b$ and the proposition follows.
\subsection{Proof of \pref{pro-supp}}
 Fix some $t_0 \geq 0$; we will show that $\textnormal{ess\,inf}(m(t))\geq \textnormal{ess\,inf}(m(t_0))$ for every $t \geq t_0$. Clearly, it is sufficient to consider the case
 $\textnormal{ess\,inf}(m(t_0))> 0$. Take some arbitrary $a < \textnormal{ess\,inf}(m(t_0))$. Let $h(x)=\ind{x \leq a}$ and $\varphi(t)=\cro{h,m(t)}$. We have $\varphi(t_0)=0$ and, by definition of Problem~1:
\bearn
  \varphi(t) &\leq& 2 \int_{t_0}^t\cro{\abs{h(wx+(1-w)y)-h(x)},m(s)(dx)m(s)(dy)}ds
\eearn
Note that $\abs{h(wx+(1-w)y)-h(x)}\leq 1$ and that $h(wx+(1-w)y)-h(x)\neq 0$ requires either $x\leq a,y>a$ or $x>a, y \leq a$. Thus
 \bearn
  \varphi(t) &\leq& 2 \int_{t_0}^t 2 \varphi(s)(1-\varphi(s))ds \leq 4 \int_{t_0}^t \varphi(s)ds
\eearn
By Gr\"onwall's lemma, this shows that $\varphi(t)= 0$
for $t\geq t_0$. Thus $m(t)[0,a]= 0$ for all $t\geq t_0$ and
this is true for any $a<\textnormal{ess\,inf}(m(t_0))$ thus
$\textnormal{ess\,inf}(m(t)) \geq
\textnormal{ess\,inf}(m(t_0))$. This shows
$\textnormal{ess\,inf}(m(t))$ is non decreasing. The proof is
similar by analogy for the $\textnormal{ess\,sup}$.

\subsection{Proof of \thref{theo-cv-pc}}
%{\red The proof is for convergence to a partial consensus for
%$\Delta' < \Delta$, see remark before \thref{theo-cv-pc}.}

1. We show that $m(t)$ converges to some probability
$m(\infty)$. This follows from \pref{lem-convex} applied for
example to the family of functions $h_{\omega}: x\to e^{-\omega
x}$ indexed by $\omega \in [0,\infty)$. For any fixed $\omega$,
$\cro{h_{\omega},m(t)}$ is a nondecreasing function of $t$ and
is non-negative, thus converges as $t\to \infty$. The limit is a
probability (apply convergence to the constant equal to $1$).

\noindent 2. We would like to conclude that $m(\infty)$ is a
stationary point, i.e. $\cro{\calA(m(\infty))h,m(\infty)}=0$
for any $h \in L^{\infty}[0,1]$, however there is a technical
difficulty since the definition of $\calA$ involves the
discontinuous function $\ind{\abs{x-y}\leq \Delta}$. We circumvent
the difficulty as follows. For $\eps>0$ and smaller than
$\Delta$, let $\ell_{\eps}(x)$ be the continuous function of $x
\in \Reals^+$ equal to $1$ for $x \leq \Delta-\eps$, $0$ for $x
\geq \Delta$, and the linear interpolation in-between. We have
$\ind{x \leq \Delta-\eps}\leq \ell_{\eps}(x)\leq \ind{x \leq
\Delta}$ for all $x \geq 0$. Let $h(x)=x^2$. By the alternative
definition of Problem~1, for $t$ and $u\geq 0$:
 \bearn
 \lefteqn{\cro{h,m(t+u)}-\cro{h,m(t)} }\\& \leq & - 2 w(1-w)
  \int_{t}^{t+u}
    \cro{(x-y)^2
     \ell_{\eps}(\abs{x-y})
    ,
    m(s)(dx)m(s)dy
    }ds
 \eearn
 Fix $u\geq 0$ and let $t \to \infty$. By weak convergence of
 the product measure $m(t)\otimes m(t)$ it follows that
 \bearn
 0 & \leq &- 2 w(1-w) u\cro{(x-y)^2
     \ell_{\eps}(\abs{x-y})
    ,
    m(\infty)(dx)m(\infty)dy
    }
 \eearn
and thus $
 \cro{(x-y)^2
     \ell_{\eps}(\abs{x-y})
    ,
    m(\infty)(dx)m(\infty)dy
    }=  0
 $ from where we conclude that
 \be
  \cro{(x-y)^2
     \ind{\abs{x-y}\leq \Delta-\eps}
    ,
    m(\infty)(dx)m(\infty)dy
    }=  0 \label{eq-pr-kjs}
 \ee for all $\eps \in (0, \Delta)$.

%3. %i.e. $ (x-y)^2\ell_{\eps}(\abs{x-y})=0$, $m(\infty)\otimes
%m(\infty)$ a.s.; since $\ind{\abs{x-y}\leq \Delta-\eps} \leq
%\ell_{\eps}(\abs{x-y})$, $(x-y)^2\ind{\abs{x-y}\leq
%\Delta-\eps}=0$, $m(\infty)\otimes m(\infty)$ a.s.  as well.
%This is true for all $\eps>0$, $\eps <\Delta$, thus
% \ben
% \cro{(x-y)^2\ind{\abs{x-y}<\Delta}=0,
% m(\infty)(dx)m(\infty)(dy)} =0
% \een
 \noindent 3. Fix some $\eps >0$ and  integrate the previous equation with respect to $y$;
  it comes that
 $\cro{r(x),m(\infty)(dx)}=0$ with
 $r(x)\eqdef\cro{(y-x)^2\ind{\abs{y-x}\leq \Delta-\eps},m(\infty)(dy)}$,
 thus there is a set $\Omega_1\subset [0,1]$ with $m(\infty)(\Omega_1)=1$
 and $r(x)=0$ for every $x \in \Omega_1$. Let $x_1$ be an
 element of $\Omega_1$ (which is not empty since $m(\infty)(\Omega_1)=1$).
 Then $r(x_1)=0$ and thus $m(\infty)\lp\lb(x_1 -\Delta+\eps,x_1) \cup (x_1,
 x_1 + \Delta-\eps)\rb
 \cap [0,1] \rp=0$ and the restriction of $m(\infty)$ to
 $(x_1-\Delta+\eps, x_1+\Delta-\eps)\cap [0,1]$ is a dirac mass at $x_1$.
 Apply the same reasoning to the complement of $(x_1-\Delta+\eps,
 x_1+\Delta-\eps)$, this shows recursively that $m(\infty)$ is a
 finite sum of Dirac masses, i.e.
 $m(\infty)=\sum_{i=1}^I \alpha_i \delta_{x_i}$
 for some $I \in \Nats$, $\alpha_i >0$, $\sum_{i=1}^I\alpha_i=1$ and $x_i \in [0,1]$.

Assume that $\abs{x_i-x_j}<\Delta$ for some $i\neq j$. Apply
\eref{eq-pr-kjs} with $\eps=\frac{\Delta-\abs{x_i-x_j}}{2}$.
The right-handside of \eref{eq-pr-kjs} is lower bounded by
$\alpha_i \alpha_j (x_i-x_j)^2>0$, which is a contradiction.
Therefore $\abs{x_i-x_j}\geq\Delta$ for all $i\neq j$.

\subsection{Proof of \pref{prop-qdec}}
First we show that if $\nu\in P_{n+1}(\mu_0)$
then there exists some $\nu'\in P_{n}(\mu_0)$
with $\cro{h,\nu'} \leq \cro{h,\nu}$, which will
clearly show the proposition.

We are given $\nu=\sum_{i=1}^{n+1}\alpha_i
\delta_{x_i} \in P_{n+1}(\mu_0)$. Let
$x'_n=\frac{\alpha_{n}x_n +\alpha_{n+1}x_{n+1}
}{\alpha_n+\alpha_{n+1}}$ and
 \ben
   \nu'=
   \sum_{i=1}^{n-1}\alpha_i \delta_{x_i} +
   \lp
     \lp \alpha_n+\alpha_{n+1}\rp \delta_{x'_n}
   \rp
 \een
We have $\nu'\in P_n(\mu)$ and by convexity of
$h$:
 \ben
  \lp \alpha_n+\alpha_{n+1}\rp h(x'_n) \leq
  \alpha_n h(x_n) + \alpha_{n+1}h(x_{n+1})
 \een thus $\cro{h,\nu'} \leq \cro{h,\nu}$ as
 required.

\subsection{Proof of \thref{theo-ccbq}}
By hypothesis $\cro{h,m_0}\leq q$ and since $h$ is continuous,
by \thref{theo-cv-pc}, $\cro{h,m(\infty)}\leq q$. Since the
mean of $m(\infty)$ is also $\mu_0$ (again by
\thref{theo-cv-pc} applied to $h(x)=x$), it follows that $q$ is
not in $Q_{d}(h,\mu_0)$. Together with the hypothesis $q \in
Q_{n}(h,\mu_0)$, \pref{prop-qdec} implies that $c<n$.

\subsection{Proof of \pref{prop-invalsym}}
Let $m'(t)$ be the image measure of $m(t)$ by $x \mapsto 1-x$.
By direct computation and the alternative form of Problem~1, it
follows that $m'(t)$ is solution to Problem~1 with initial
condition $m'(0)=m(0)$. By uniqueness, $m'(t)=m(t)$.

\subsection{Proof of \pref{prop-qdec-s}}
Let $\nu$ be a symmetric partial consensus with $n+1$
components. We do as in the proof of \pref{prop-qdec}: If $n+1$
is even, we replace the two middle components by their weighted
averages. If $n+1$ is odd, we replace the three middle
components $x_{m-1}, x_m=0.5, x_{m+1}$ (with $m=n/2 +1$) by two
components $(\alpha_{m-1}x_{m-1}+0.5 \alpha_m
x_m)/(\alpha_{m-1}+0.5\alpha_{m})$
   and
   $(0.5\alpha_{m}x_{m}+
\alpha_{m+1} x_{m+1})/(0.5\alpha_{m}+ \alpha_{m+1})$ with
weights $\alpha_{m-1}+0.5\alpha_{m}$ and $0.5\alpha_{m}+
\alpha_{m+1}$.  We obtain some $\nu' \in SP_{n}$ and
$\cro{h,\nu'}\leq \cro{h,\nu}$ for any convex $h$, thus if $q
\in SQ_n(h)$ we must also have $q \in SQ_{n+1}(h)$.

\subsection{Proof of \thref{theo-ccbq-s}}
The proof is similar to \thref{theo-ccbq}.
%
%\subsection{Proof of \sref{sec-num} Proofs}
%\subsubsection*{Proof of \pref{pro-equivp1p2}}

\subsection{Proof of \thref{theo-pdf}}

Assuming that $m_0$ is absolutely continuous, the fact that $m(t)$
is absolutely continuous can be proved
by probabilistic arguments which use representations by inhomogeneous
Markov processes with uniformly bounded jump rates.

More precisely, the proof of Theorem 2.1 in
Desvillettes \emph{et al}.,\cite{desvillettes1999probabilistic}
for a class of equations (the generalized cutoff Kac equation)
with the same probabilistic structure as ours,
extends immediately to the present
situation. It is an extension of \thref{theo-eunl} proved using only its hypotheses.

If $m=(m(t), {t\in\Reals_+})$ is a solution of Problem 1 and
$m(t)(dx) = f(x,t)\,dx$ then, for any bounded $h$,
an elementary change of variables yields
\begin{align*}
&
\int h(x) f(x,t)\,dx
- \int h(x) f(x,0)\,dx
\\
&\qquad =
2\int_0^t \iint h(w x + (1-w) y) \ind{|x-y|\le\Delta}
f(x,s) f(y,s)
\,dxdy\,ds
\\
&\qquad\qquad
- 2\int_0^t \iint h(x) \ind{|x-y|\le\Delta}
f(x,s) f(y,s)
\,dxdy\,ds
\\
&\qquad =
\frac{2}{w}\int_0^t \int h(x')
\biggl[\int_{x'-\Delta w}^{x'+\Delta w}
f\biggl(\frac{x'-(1-w)y}{w},s\biggr) f(y,s)
\,dy\biggr]\,dx' \,ds
\\
&\qquad\qquad
- 2\int_0^t \int h(x)f(x,s) \biggl[\int_{x-\Delta}^{x+\Delta} f(y,s)\,dy\biggr]\,dx\,ds
\end{align*}
from which \eqref{eq-pdf} readily follows.

The converse statement follows by integrating \eref{eq-pdf} by $h(x)\,dx$,
which after the reverse change of variables yields Problem 1 as a weak formulation.

\eref{eq-pdf-bis} is obtained similarly using the change of
variables $x'=\frac{w x - (1-w)y}{2w-1}$ and $y'=\frac{w y -
(1-w)x}{2w-1}$.

\subsection{Proof of \pref{pro-fbounded}}

Since $f(x,t)$ is non-negative, we have:

$$\frac{\partial f(x,t)}{\partial t}
\leq \frac{2}{w}\int_{x-w\Delta}^{x+w\Delta}f(y,t)f\left(\frac{x-(1-w)y}{w},t\right)dy.$$

For a fixed arbitrary $t$, let $A_i = \{x \in \text{ Supp}(f(x,t)) | i-1 < f(x,t) \leq i\}, \; i > 0$ be the level sets. Note that $A_j = \emptyset$ for all $j > \lceil M(t) \rceil$ and that the $A_i$ are disjoint. For any $x$, we have that:
$$ \frac{2}{w}\int_{x-w\Delta}^{x+w\Delta}f(y,t)f\left(\frac{x-(1-w)y}{w},t\right)dy $$
$$\leq \frac{2}{w}\sum_{i,j} \mu\left(\left\{y\left|y \in A_i, \frac{x - (1-w)y}{w}\in A_j\right.\right\}\right)\max{\{i,j\}^2}.$$

Using the fact that the $A_i$ are disjoint we can get that:

\bearn \lefteqn{\frac{2}{w}\sum_{i,j} \mu\left(\left\{y\left|y
\in A_i, \frac{x - (1-w)y}{w}\in
A_j\right.\right\}\right)\max{\{i,j\}^2}}
 \\
& =& \frac{2}{w}\sum_{i} \mu\left(\left\{y\left|y \in A_i,
\frac{x - (1-w)y}{w}\in \bigcup_{k \leq i}
A_k\right.\right\}\right)i^2
 \\
 &&+  \frac{2}{w}\sum_{i} \mu\left(\left\{y\left|y \in
\bigcup_{k < i} A_k, \frac{x - (1-w)y}{w}\in
A_i\right.\right\}\right)i^2 = I_1 + I_2. \eearn
We can bound $I_1$ and $I_2$ now as:
$$ I_1 \leq \frac{2}{w}\sum_{i}\mu(A_i) i^2, \qquad  I_2 \leq \frac{2}{1-w}\sum_{i}\mu(A_i) i^2,$$
subject to the following restrictions:
$$ \sum_{i}\mu(A_i) \leq 1, \qquad \sum_{i}(i-1)\mu(A_i) \leq \int_{0}^{1}f(x,t)dx = 1.$$
Plugging the second restriction into the bound of $I_1$ and $I_2$, we get that:
$$ \sum_{i=1}^{\lceil M(t) \rceil}\mu(A_i) i^2 \leq \frac{\lceil M(t) \rceil ^2}{\lceil M(t) \rceil-1} + \sum_{i=1}^{\lceil M(t) \rceil-1}\mu(A_i)\left(i^2 - \frac{\lceil M(t) \rceil ^2}{\lceil M(t) \rceil-1}(i-1)\right)$$
$$
= \frac{\lceil M(t) \rceil ^2}{\lceil M(t) \rceil-1} + \frac{1}{\lceil M(t) \rceil-1}\sum_{i=1}^{\lceil M(t) \rceil-1}\mu(A_i)\left(\lceil M(t) \rceil i - \lceil M(t) \rceil - i) (i-\lceil M(t) \rceil\right).$$
The maximum of the RHS is attained when $\mu(A_i) = 0 \quad \forall \;i > 1$ and $\mu(A_1)$ is as big as possible. By the first restriction, $\mu(A_1) = 1$. In that case, we have that:
$$ \sum_{i=1}^{\lceil M(t) \rceil}\mu(A_i) i^2 \leq
\frac{\lceil M(t) \rceil ^2}{\lceil M(t) \rceil-1} + 1
\leq \lceil M(t) \rceil + 3 \leq M(t) + 4.$$
Therefore:
$$ \sup_{A_i}\left\{\sum_{i}\mu(A_i) i^2\right\} \leq M(t) + 4.$$
Finally, for any $x$ we have:
$$ \frac{2}{w}\int_{x-w\Delta}^{x+w\Delta}f(y,t)f\left(\frac{x-(1-w)y}{w},t\right)dy \leq I_1 + I_2 \leq \left(\frac{2}{w} + \frac{2}{1-w}\right)(M(t) + 4), $$
which means that:
$$ M'(t) \leq \left(\frac{2}{w} + \frac{2}{1-w}\right)(M(t) + 4).$$
Integrating, we get the result.
%
%following bound:

%$$ M(T) \leq e^{\left(\frac{2}{w} + \frac{2}{1-w}\right)T}(M(0) + 4) - 4$$

%which proves the result.

\subsection{Proof of \pref{pro-absfbounded}}

Again, since $f(x,t)$ is non-negative, for all $x$,
 \bearn
 \lefteqn{
\left|\frac{\partial}{\partial t}f(x,t)\right|}\\
 &\leq&
\max{\left\{\frac{2}{w}\int_{x-w\Delta}^{x+w\Delta}f(y,t)f\left(\frac{x-(1-w)y}{w},t\right)dy,\;2f(x,t)\left(\int_{x-\Delta}^{x+\Delta}f(y,t)dy\right)\right\}}.
 \eearn
On the one hand,
$$ 2f(x,t)\left(\int_{x-\Delta}^{x+\Delta}f(y,t)dy\right) \leq 2M(t)\int_{0}^{1}f(y,t)dy \leq 2M(t),$$
on the other, using \pref{pro-fbounded},
$$ \frac{2}{w}\int_{x-w\Delta}^{x+w\Delta}f(y,t)f\left(\frac{x-(1-w)y}{w},t\right)dy \leq \left(\frac{2}{w} + \frac{2}{1-w}\right)(M(t) + 4), $$
therefore
$$ \left|\frac{\partial}{\partial t}f(\cdot,t)\right|_{\infty} \leq \left(\frac{2}{w} + \frac{2}{1-w}\right)(M(t) + 4).$$

%\subsection{Proof of \pref{prop-opt-piece}}
%
%As $f^{e}(x,t+\Delta t)$ is piecewise linear, we can treat each interval independently.
%Given a $\nu_e(x)$ associated to $f^e(x) = ax+b$ we want to find:
%$$ \min_{\nu_r}\int_{X}|d\nu_e(x) - d\nu_r(x)| = \min_{M}\int_{x_s}^{x_e}|ax+b-M|dx.$$
%
%If $a = 0$, then $M = b$ has zero error. Let's suppose $a \neq 0$. If $M$ lies between $ax_s+b$ and $ax_e + b$, then:
%$$ \min_{M}\int_{x_s}^{x_e}|ax+b-M|dx = \min_{M}\frac{1}{2a}[(ax_e+b-M)^2 + (ax_s+b-M)^2] =
%\frac{a}{4}(x_e - x_s)^2.$$
%
%The minimum is attained for
%$$ M_{min} = \frac{\displaystyle \int_{x_s}^{x_e}(ax+b) dx}{x_e - x_s} = \frac{a}{2}(x_s + x_e) + b,$$
%
%which is the value of the function at the midpoint of the interval. If $M$ lies outside $ax_s + b$ and $ax_e + b$, then the error is greater than the %previous case as we could minimize it by setting $M$ to one of the extremal values of $f^{e}$ in $X$.

\subsection{Proof of \pref{prop-num-integral}}
Let $f^r(x,t)$ be defined piecewise in the intervals $X_i = [x_i,x_{i+1}]$ and let $M_{i}$ be %the value of $M$ that minimizes the error for the interval $X_i$.
constant chosen for the piecewise constant approximation on the interval $X_{i}$.
We have that, independently of $t$:
$$ \int_{0}^{1}f^r(x,t)dx = \int_{0}^{1} \sum_{i=1}^{I}M_{i}1_{X_i}dx
= \sum_{i=1}^{I}\int_{X_i}\frac{\displaystyle \int_{x_i}^{x_{i+1}}f^e(y,t)dy}{x_{i+1} - x_i}dx
= \int_{0}^{1}f^e(y,t)dy.$$
%
%\color{blue}
\subsection{Proof of \pref{prop-err_ecs}}

%We first calculate the error when $I = I_0$.
Keeping in mind that for any interval, the slope of $f^e(x,k\Delta t)$ is bounded by $\frac{2\Delta t |\partial_t f^r(x,(k-1)\Delta t)|_{\infty}}{1/I}$, yielding:
\begin{equation}
\label{eq-ecs}
\varepsilon_{c.s.}(I) \leq I\frac{\text{Max. Slope}}{2}\left(\frac{1}{I}\right)^2 = \Delta t|\partial_t f^r(x,(k-1)\Delta t)|_{\infty}.
\end{equation}
%
%However, if we divide each interval in two, the error is halved, because the error with two intervals equals $ \displaystyle 2\frac{\text{Slope}}{4}\left(\frac{1}{1/2I_0}\right)^2$, where with one is equal to $\displaystyle \frac{\text{Slope}}{4}\left(\frac{1}{1/I_0}\right)^2$. Therefore, for sufficiently large $I$ we can write:
%\begin{equation}
%\varepsilon_{c.s}(I) = \varepsilon_{c.s}(I_0)\frac{I_0}{I} \leq \frac{\Delta t}{2}|\partial_t f^r(x,(k-1)\Delta t)|_{\infty} \frac{I_0}{I}.
%\end{equation}
%
%\subsection{Proof of \pref{prop-const_spline_deriv_bound}}
%
On the other hand:
$$ M(\Delta t) = |f^r(x,\Delta t)|_{\infty} \leq |f^e(x,\Delta t)|_{\infty} \leq |f^r(x,0)|_{\infty} + \Delta t |\partial_t f^r(x,0)|_{\infty} $$
$$\leq M + \Delta t K_1 M + \Delta t K_2 = (1 + \Delta t K_1)M + \Delta t K_2,$$
where $K_1 = \frac{2}{w} + \frac{2}{1-w}, K_2 = \frac{8}{w} +
\frac{8}{1-w}$. The first inequality is true because when we
approximate by piecewise constant splines, the maximum of the
function decreases and the third is true by
\pref{pro-fbounded}. Note that in order to be able to apply it
we are implicitly using \pref{prop-num-integral} as the total
mass is conserved. By induction:
$$ M\left(\frac{T}{\Delta t}\Delta t\right) \leq (1 + \Delta t K_1)^{\frac{T}{\Delta t}}M + \Delta t K_2 \sum_{i=0}^{T/\Delta t-1}(1+\Delta t K_1)^i$$
$$ = (1+\Delta t K_1)^{\frac{T}{\Delta t}}M + \frac{K_2}{K_1}((1+\Delta t K_1)^{\frac{T}{\Delta t}} - 1) \leq \frac{K_2}{K_1}(1+\Delta t K_1)^{\frac{T}{\Delta t}}\left(M + \frac{K_2}{K_1}\right).$$
We can now bound $M(k\Delta t)$ in the following way. As $K_1$
and $K_2$ are positive, taking into account that $(1+K_1 \Delta
t)^{\frac{T}{\Delta t}}$ is decreasing with $\Delta t$, we have
for any $k$:
$$ M(k\Delta t) \leq (1 + \Delta t K_1)^{\frac{T}{\Delta t}}M + \frac{K_2}{K_1}(1+\Delta t K_1)^{\frac{T}{\Delta t}} \leq
e^{K_1 T}\left(M + \frac{K_2}{K_1}\right).$$
Using \pref{pro-absfbounded}:
$$ |\partial_t f^r(x,(k-1)\Delta t)|_{\infty} \leq K_1 M((k-1)\Delta t) + K_2 \leq K_1 e^{K_1 T}\left(M + \frac{K_2}{K_1}\right) + K_2 \equiv c.$$
Combining this equation with equation \eqref{eq-ecs} we get the desired result.
%\color{black}
\subsection{Proof of \pref{prop-error-euler}}
We have that:
\bearn \varepsilon_{eu} &= &|\nu_{e}^{k \Delta t}(dx) -
\mu_{(k-1)\Delta t}^{k\Delta t}(dx)|_{T}
 \\
& = & \int_{0}^{1}|g^{(k-1)\Delta t}(x,k\Delta t) -
g^{(k-1)\Delta t}(x,(k-1)\Delta t) - \Delta t \partial_t
g^{(k-1)\Delta t}(x,(k-1)\Delta t)|dx
 \\
 & \leq &\frac{1}{2}(\Delta t)^2
|\partial_{tt}^{2}g^{(k-1)\Delta t}(x,(k-1)\Delta t)|_{\infty}
+ O\left((\Delta t)^3\right).
 \eearn
By \coref{cor-derivatives_bounded}, we can bound, for any $k$:
$$ |\partial_{tt}^{2}g^{(k-1)\Delta t}(x,(k-1)\Delta t)|_{\infty} \leq 16\Delta
|\partial_{t}g^{(k-1)\Delta t}(x,(k-1)\Delta t)|_{\infty}
|g^{(k-1)\Delta t}(x,(k-1)\Delta t)|_{\infty}$$
$$
\leq 16\Delta \left(K_1 e^{K_1 T}\left(M + \frac{K_2}{K_1}\right) + K_2\right)e^{K_1 T}\left(M + \frac{K_2}{K_1}\right) = C_2,$$
therefore:
\begin{equation}
 \varepsilon_{eu} \leq \frac{C_2}{2}(\Delta t)^2 + O((\Delta t)^3) = O((\Delta t)^2).
\end{equation}
\subsection{Proof of \pref{prop-error-continuation}}
\bearn
 \lefteqn{\frac{\partial}{\partial t}\int_{0}^{1}|g^{k\Delta
t}(x,t) - g^{(k-1)\Delta t}(x,t)|dx  \leq
\int_{0}^{1}|\partial_t g^{k\Delta t}(x,t) - \partial_t
g^{(k-1)\Delta t}(x,t)|}
 \\
&\leq& \int_{0}^{1}2\left|-g^{k\Delta
t}(x,t)\int_{x-\Delta}^{x+\Delta}g^{k\Delta t}(y,t)dy +
g^{(k-1)\Delta t}(x,t)\int_{x-\Delta}^{x+\Delta}g^{(k-1)\Delta
t}(y,t)dy\right|
  \\
&&+
\int_{0}^{1}\frac{2}{w}\left|\int_{x-w\Delta}^{x+w\Delta}g^{k\Delta
t}(y,t)g^{k\Delta t}\left(\frac{x-(1-w)y}{w},t\right)dy\right.
 \\
 &&
- \left.\int_{x-w\Delta}^{x+w\Delta}g^{(k-1)\Delta
t}(y,t)g^{(k-1)\Delta
t}\left(\frac{x-(1-w)y}{w},t\right)dy\right| = I + J.
 \eearn
We will first bound $I$. We have that:
$$I \leq 2\int_{0}^{1}|g^{(k-1)\Delta t}(x,t)-g^{k\Delta t}(x,t)|\int_{x-\Delta}^{x+\Delta}g^{(k-1)\Delta t}(y,t)dzdx $$
$$+ 2\int_{0}^{1}g^{k\Delta t}(x,t)\int_{x-\Delta}^{x+\Delta}|g^{(k-1)\Delta t}(y,t)-g^{k\Delta t}(y,t)|dzdx = I_1 + I_2.$$
On the one hand:
$$ I_1 \leq 2\int_{0}^{1}|g^{(k-1)\Delta t}(x,t)-g^{k\Delta t}(x,t)|dx,$$
on the other:
 \begin{align*}
 I_2 & \leq 2\int_{0}^{1}g^{k\Delta t}(x,t)\int_{0}^{1}|g^{(k-1)\Delta t}(y,t)-g^{k\Delta t}(y,t)|dzdx \\
 & \leq 2\int_{0}^{1}|g^{(k-1)\Delta t}(x,t)-g^{k\Delta t}(x,t)|dx.
 \end{align*}
Now we will bound $J$:
 \begin{align*}
 J & \leq \frac{2}{w}\int_{0}^{1}\int_{x-w\Delta}^{x+w\Delta}g^{k\Delta t}(y,t) \\
 & \times \left|g^{k\Delta t}\left(\frac{x-(1-w)y}{w},t\right)-g^{(k-1)\Delta t}\left(\frac{x-(1-w)y}{w},t\right)\right|dzdx \\
 & + \frac{2}{w}\int_{0}^{1}\int_{x-w\Delta}^{x+w\Delta} \left|g^{k\Delta t}(y,t) - g^{(k-1)\Delta t}(y,t)\right| \\
 & \times g^{(k-1)\Delta t}\left(\frac{x-(1-w)y}{w},t\right)dzdx = J_1 + J_2
 \end{align*}
$$ J_1 = 2\int_{0}^{1}\int_{x-\Delta}^{x+\Delta}g^{k\Delta t}(x,t)\left|g^{k\Delta t}\left(y,t\right)-g^{(k-1)\Delta t}\left(y,t\right)\right|dzdx $$
$$\leq 2\int_{0}^{1}|g^{k\Delta t}(x,t)-g^{(k-1)\Delta t}(x,t)|dx$$
$$ J_2 = 2\int_{0}^{1}\int_{x-\Delta}^{x+\Delta}\left|g^{k\Delta t}(x,t) - g^{(k-1)\Delta t}(x,t)\right|g^{(k-1)\Delta t}\left(y,t\right)dzdx $$
$$\leq 2\int_{0}^{1}|g^{k\Delta t}(x,t)-g^{(k-1)\Delta t}(x,t)|dx.$$
Adding all the equations together we get that:
$$ \frac{\partial}{\partial t}\int_{0}^{1}|g^{k\Delta t}(x,t) - g^{(k-1)\Delta t}(x,t)|dx \leq I + J \leq I_1 + I_2 + J_1 + J_2 $$
$$ \leq 8\int_{0}^{1}|g^{k\Delta t}(x,t)-g^{(k-1)\Delta t}(x,t)|dx.$$

Integrating:
$$ \left|\mu_{k\Delta t}^{t}(dx) - \mu_{(k-1)\Delta t}^{t}(dx)\right|_{T} = \int_{0}^{1}|g^{k\Delta t}(x,t) - g^{(k-1)\Delta t}(x,t)|dx $$
$$\leq e^{8(t-k\Delta t)}\int_{0}^{1}|g^{k\Delta t}(x,k\Delta t) - g^{(k-1)\Delta t}(x,k\Delta t)|dx $$
$$ = e^{8(t-k\Delta t)}|\mu_{k\Delta t}^{k \Delta t}(dx) - \mu_{(k-1)\Delta t}^{k\Delta t}(dx)|_{T},$$
as we wanted to prove.

%\color{red}
\subsection{Proof of \thref{theo-order-method}}
\bearn
 \varepsilon_{tot} &\leq&
\sum_{k=1}^{T/(\Delta t)}\left|\mu_{(k-1)\Delta t}^{T}(dx) - \mu_{k\Delta t}^{T}(dx)\right|_{T}
\leq
e^{8T} \sum_{k=1}^{T/(\Delta t)}\left|\mu_{(k-1)\Delta t}^{k\Delta t}(dx)
 - \mu_{k\Delta t}^{k\Delta t}(dx)\right|_{T}
\\
&=& e^{8T}\frac{T}{\Delta t}\left(c \Delta t + O\left((\Delta t)^2\right)\right) = C + O\left(\Delta t\right).
\eearn
%\color{black}
%
%\color{red}
%\LARGE
%REFERENCES 6 AND 19 SEEM A BIT ODD. IS IT NORMAL TO WRITE A REFERENCE THAT POINTS TO ANOTHER REFERENCE?
%\normalsize
%\color{black} 

\bibliographystyle{ws-acs}
\bibliography{reputation}

\end{document}